\documentclass[final,onefignum,onetabnum]{siamart171218}

\usepackage{amsmath,amssymb}
\usepackage{algorithm,algorithmic}
\usepackage{graphicx}
\usepackage{textcomp}
\usepackage{amsfonts}
\usepackage{graphicx}
\usepackage{epstopdf}
\usepackage{algorithmic}
\usepackage{amsopn}
\usepackage{color}
\usepackage{lipsum}
\usepackage{epstopdf}
\usepackage{amsopn}
\usepackage{enumerate}
\usepackage{cancel}
\usepackage{mathrsfs}
\usepackage{mathdots}
\usepackage{euscript}
\usepackage{amscd}
\usepackage{placeins}

\usepackage{tikz}
\usetikzlibrary{arrows}
\usepackage{verbatim}

\graphicspath{{images/}}

\newcommand{\nn}{\nonumber}

\newcommand{\bmat}{\left[ \begin{matrix}}
\newcommand{\emat}{\end{matrix} \right]}
\newcommand{\innerprod}[2]{\langle{#1},\,{#2}\rangle}

\DeclareMathOperator{\Ebb}{{\mathbb E}}
\newcommand{\Rbb}{\mathbb R}
\newcommand{\Hbb}{\mathbb H}

\newcommand{\Dbb}{\mathbb D}

\newcommand{\Nbb}{\mathbb N}
\newcommand{\Zbb}{\mathbb Z}

\newcommand{\Tbb}{\mathbb T}

\newcommand{\yb}{\mathbf  y}
\newcommand{\zb}{\mathbf  z}

\newcommand{\fb}{\mathbf  f}

\newcommand{\ab}{\mathbf a}
\newcommand{\bb}{\mathbf  b}

\newcommand{\pb}{\mathbf  p}
\newcommand{\mb}{\mathbf  m}

\newcommand{\cb}{\mathbf c}
\newcommand{\qb}{\mathbf q}
\newcommand{\tb}{\mathbf t}
\newcommand{\kb}{\mathbf k}
\newcommand{\lb}{\boldsymbol{\ell}}
\newcommand{\oneb}{\mathbf 1}
\newcommand{\zerob}{\mathbf 0}

\newcommand{\Nb}{\mathbf N}

\newcommand{\thetab}{\boldsymbol{\theta}}
\newcommand{\phib}{\boldsymbol{\varphi}}

\newcommand{\zetab}{\boldsymbol{\zeta}}

\newcommand{\Pfrak}{\mathfrak{P}}

\newcommand{\Cfrak}{\mathfrak{C}}

\newcommand{\Cscr}{\mathscr{C}}

\newcommand{\Mscr}{\mathscr{M}}

\newcommand{\Lcal}{\mathcal{L}}
\newcommand{\Zcal}{\mathcal{Z}}

\newcommand{\Sc}{\mathcal{S}_+(\Tbb^d)}
\newcommand{\Scb}{\mathcal{S}(\Tbb^d)}
\renewcommand{\d}{\mathrm{d}}

\newcommand{\m}{\mu} 
\newcommand{\const}{\text{const.}}
\newcommand{\half}{\text{half}}

\ifpdf
  \DeclareGraphicsExtensions{.eps,.pdf,.png,.jpg}
\else
  \DeclareGraphicsExtensions{.eps}
\fi

\newsiamremark{remark}{Remark}
\newsiamremark{hypothesis}{Hypothesis}
\crefname{hypothesis}{Hypothesis}{Hypotheses}
\newsiamthm{claim}{Claim}
\newsiamthm{assumption}{Assumption}
\newsiamthm{conjecture}{Conjecture}
\newsiamthm{fact}{Fact}

\newsiamremark{example}{Example}
\newsiamremark{problem}{Problem}
 
\newcommand{\alg}[1]{\begin{align} #1 \end{align}}

\headers{Multidimensional Covariance and Generalized Cepstral Extension}{B. Zhu and M. Zorzi}

\title{A Well-Posed Multidimensional Rational Covariance and Generalized Cepstral Extension Problem\thanks{Submitted to the editors October 14, 2021, revised April 23 and December 30, 2022.
\funding{This work was supported in part by the National Natural Science Foundation of China under the grant number 62103453, the ``Hundred Talent'' Program of Sun Yat-sen University, and the SID project ``A Multidimensional and Multivariate Moment Problem Theory for Target Parameter Estimation in Automotive Radars'' (ZORZ\_SID19\_01) funded by the Department of Information Engineering, University of Padova.}}}

\author{
	Bin Zhu\thanks{School of Intelligent Systems Engineering, Sun Yat-sen University, Gongchang Road 66, 518107 Shenzhen, China (\email{zhub26@mail.sysu.edu.cn}).}
	\and
	Mattia Zorzi\thanks{Department of Information Engineering, University of Padova, Via Gradenigo 6/B, 35131 Padova, Italy  (\email{zorzimat@dei.unipd.it}).}
}

\begin{document}
 
\maketitle

\begin{abstract}
In the present paper we consider the problem of estimating the multidimensional power spectral density which describes a second-order stationary random field from a finite number of covariance and generalized cepstral coefficients. The latter can be framed as an optimization problem subject to multidimensional moment constraints, i.e., to search a spectral density maximizing an entropic index and matching the moments. In connection with systems and control, such a problem can also be posed as finding a multidimensional shaping filter (i.e., a linear time-invariant system) which can output a random field that has identical moments with the given data when fed with a white noise, a fundamental problem in system identification. In particular, we consider the case where the dimension of the random field is greater than two for which a satisfying theory is still missing. We propose a multidimensional moment problem which takes into account a generalized definition of the cepstral moments, together with a consistent definition of the entropy. We show that it is always possible to find a rational power spectral density matching exactly the covariances and approximately the generalized cepstral coefficients, from which a shaping filter can be constructed via spectral factorization. In plain words, our theory allows to construct a well-posed spectral estimator for any finite dimension.
\end{abstract}

\begin{keywords}
Multidimensional covariance and generalized cepstral extension, $\alpha$-divergence, trigonometric moment problem, spectral analysis.
\end{keywords}

\begin{AMS}
42A70, 30E05, 47A57, 60G12
\end{AMS}

\section{Introduction} \label{sec:introduction}

Second-order stationary random fields are characterized by multidimensional power spectral densities.
Such models are widely used in many engineering fields including systems and control, signal and image processing, etc., see e.g., \cite{bose2013multidimensional}. For instance, the measurements collected from automotive radar sensors can be modeled by a random field having a dimension $d=3$, i.e., range, velocity and azimuth dimensions \cite{ZFKZ2019fusion,van2004optimum,engels2017advances}. In those applications we have to face a spectral estimation problem which can be addressed by setting up a moment problem. More precisely, we search a multidimensional spectral density matching a finite number of moments (e.g., covariances or logarithmic moments) which are  computed from the measured data.

In the unidimensional case, i.e., $d=1$, spectral estimation problems based on covariance matching have been widely studied \cite{burg1975maximum,Kalman,musicus1985maximum,Georgiou-87,BLGM-95,BGL-98,BGL-THREE-00,Georgiou-L-03,stoica2005spectral,FMP-12,FPZ-12,Zhu-Baggio-19,Zhu-M2-LineSpec}. These paradigms search a rational spectral density which maximizes a suitable entropic index while matching a finite number of covariance lags which satisfy a positivity condition.
Then, it is also possible to embed in the problem some \emph{a priori} information by considering a divergence index \cite{FPR-08,RFP-09,Zorzi-F-12,GL-17,zhu2018wellposed,Z-TRANSP}. The majority of these divergence (and thus also the corresponding entropy) indices are connected through   
the $\alpha$-divergence \cite{Z-14rat}, $\beta$-divergence \cite{Z-14}, or $\tau$-divergence \cite{Z-15,zorzi2015interpretation}. It is important to notice that all these moment problems admit a rational solution, which leads to a finite-dimensional system (the sought-after shaping filter) after \emph{spectral factorization} {\cite{LP15}}. Moreover, it is also possible to match the logarithmic moments (also called cepstral ccoefficients {\cite{oppenheim2004frequency}}) in addition to the covariances \cite{byrnes2001cepstral}. However, it is not guaranteed that a solution does exist {to the simultaneous covariance-cepstral matching problem}. A partial remedy is to consider a regularized version of the dual optimization problem which provides a solution matching the {covariances} and approximately the cepstral coefficients \cite{enqvist2004aconvex}.

The multidimensional case has been considered recently {\cite{KLR-16multidimensional,ringh2018multidimensional,Z-MED,Liu-Zhu-2021}},
with the exception of the former works {\cite{lang1982multidimensional,lang1983spectral,ben1988dual,lev1989multidimensional,geronimo2004positive,Georgiou-06}}. The main issue is that the solution of the moment problem is not necessarily a spectral density, but rather a spectral measure that may contain a singular part \cite{RKL-16multidimensional}. In the case that only covariance matching is considered, such an issue can be addressed by considering the $\tau$-divergence family \cite{zhu2020m}: using such a divergence family, then it is guaranteed that the family of solutions to the moment problem contains at least one rational spectral density.
It is worth emphasizing the importance of rationality as a direct connection to the finite-dimensional realization theory, although spectral factorization is not always possible for positive polynomials of several variables (see \cite{geronimo2004positive,geronimo2006factorization}). Regarding the case with simultaneous covariance and cepstral matching, it is possible to regularize the dual optimization problem in the same spirit of \cite{enqvist2004aconvex}. However, the existence of a solution (i.e., a spectral density) is only guaranteed in the case {of} $d\leq 2$ \cite{RKL-16multidimensional}.

In the present paper we deal with a generalized definition of the cepstral moments, more precisely a family of cepstral moments indexed by an integer parameter $\nu$, and a consistent family of entropies indexed by $\nu$ as well. We consider the corresponding multidimensional spectral estimation problem with covariance and {generalized} cepstral matching. We derive the corresponding dual optimization problem and propose its regularized version. 
Using such a {generalized} cepstral-entropy family, it is guaranteed that the family of spectral measures corresponding to the regularized dual problem contains at least one rational spectral density that matches the covariances and approximately the {generalized} cepstral coefficients. It is important to stress that such a result holds for any $d\geq1$.

In the case that the random field is periodic, we are still able to find a solution matching the covariances and approximately the {generalized} cepstral coefficients for any $d\geq1$, see \cite{Zhu-Zorzi-21}. Such a result is not surprising because it corresponds to discretizing
the power spectral density in the frequency domain \cite{lindquist2013onthemultivariate,LPcirculant-13}, and indeed the existence of such a solution for $d>1$ is also guaranteed using the classic definitions of logarithmic moments and entropy \cite{ringh2015multidimensional,ZFKZ2019M2}. However, there is a fundamental difference between our problem and the classic one in the case of $d>2$. In the former the periodic solution approximates arbitrarily well the nonperiodic solution (provided that the period is taken sufficiently large), while in the latter this does not necessarily happen. Accordingly, our periodic formulation provides a concrete tool to compute {(approximate)} the solution of the nonperiodic problem.

The outline of the paper is as follows. In \cref{sec:classic} we review the classic covariance and cepstral extension problem, while in
\cref{sec:generalized} we introduce our generalized formulation. In \cref{sec:dual} we derive the dual optimization problem. In \cref{sec:regularization} we introduce the regularized version of the dual problem. \Cref{sec:well-posedness} shows that the regularized dual problem is well-posed in the sense of Hadamard.
\Cref{sec:circulant} considers the problem for periodic random fields whose solution approximates that in the usual nonperiodic case. Some numerical experiments are presented in \cref{sec:sysid_app}. 
Finally, in \cref{sec:conclusion} we draw the conclusions.

{\em Notation}. Let $\kb=(k_1,\ldots, k_d)$ be a vector-valued index belonging to a specific index set $\Lambda \subset \Zbb^d$ with $d\in\Nbb_+$. We shall require $\Lambda$ to have a finite cardinality, contain the all-zero index, and be symmetric with respect to the origin, namely $\kb\in\Lambda$ implies $ -\kb\in\Lambda$.
Let $\Tbb=[0, 2\pi)$ be the interval for angular frequencies and $\thetab=(\theta_1,\theta_2,\dots,\theta_d)\in \Tbb^d=[0, 2\pi)^d$. We introduce the shorthand notation $e^{i\thetab}$ for $(e^{i\theta_1},\dots,e^{i\theta_d})$ which is a point on the $d$-torus.  The functional space $L^1(\Tbb^d)$ contains absolutely integrable functions $\Phi:\Tbb^d \to \Rbb$. We conventionally write $\Phi(e^{i\thetab})$ {but} identify $\Phi$ as a function of $\thetab$ since the frequency domain $\Tbb^d$ is isomorphic to the $d$-torus. $\Scb$ denotes the set of functions in  $L^1(\Tbb^d)$ which are nonnegative, and it has a subset $\Sc$ that contains positive, coercive, and bounded functions, i.e., there exist two constants $\delta, \mu>0$ such that $\delta\leq\Phi(e^{i\thetab})\leq\mu$ for any $\thetab\in\Tbb^d$.

\section{Classic formulation}\label{sec:classic}
Let $\yb=\{\, \yb_\tb,\; \tb\in \Zbb^d\,\}$ be a zero-mean real-valued second-order stationary random field indexed by $\tb=(t_1,\ldots, t_d )\in\Zbb^d$. 
{Such a random field can be described in the frequency-domain by} the power spectral density $\Phi\in\Scb$ which can be written as a formal multidimensional Fourier series
\alg{\label{def_phi} \Phi(e^{i\thetab})=\sum_{k\in \Zbb^d} c_\kb e^{-i\innerprod{\kb}{\thetab}},}
where $c_\kb$ is the $\kb$-th covariance coefficient defined as $c_\kb :=\Ebb[\yb_\tb\yb_{\tb+\kb}]$, independent of $\tb$ due to stationarity, and $\innerprod{\kb}{\thetab} := k_1\theta_1+\cdots+k_d\theta_d$ is the usual inner product in $\Rbb^d$. The symmetry $c_{\kb}=c_{-\kb}$ follows easily from the definition of the covariance\footnote{{Notice that the theory developed in this paper holds almost verbatim for complex-valued random fields whose covariances have conjugate symmetry $c_{\kb}=c_{-\kb}^*$.}}. {Moreover, if the random field happens to be Gaussian, then it is completely characterized by the spectral density $\Phi$.}

Assume further that $\Phi$ is strictly positive on $\Tbb^d$ and the log-spectrum also admits a Fourier series (which is the case when e.g., $\Phi\in\Sc$)
\alg{ \log\Phi(e^{i\thetab})=\sum_{k\in \Zbb^d} m_\kb e^{-i\innerprod{\kb}{\thetab}},}
where $m_\kb=m_{-\kb}$ denotes the $\kb$-th  cepstral coefficient. In many applications we can measure a finite-size multisequence extracted from a realization of $\yb$ from which we can estimate a finite number of covariance and cepstral coefficients. More precisely, we choose an index set $\Lambda$ and estimate $c_\kb$ with $\kb\in \Lambda$ and $m_{\kb}$ with $\kb\in\Lambda_0:=\Lambda\setminus \{\zerob\}$\footnote{The cepstral coefficient $m_{\zerob}$ is deliberately excluded for technical reasons.}. Typically, we can take
$\Lambda=\{\, \kb=(k_1\ldots k_d )\in\Zbb^d \; : \; |k_j|\leq n_j, \; j=1,\ldots, d\,\}$,
where each integer $n_j$ is sufficiently small relative to the data length in the $j$-th dimension.
Then, we would estimate the power spectral density of $\yb$  from those $c_\kb$'s and $m_\kb$'s . The latter is a spectral estimation problem which can be formulated as follows.

\begin{problem}\label{pb_orig}
Given two sets of real numbers $\cb=\{c_\kb\}_{\kb\in\Lambda}$ such that $c_\kb=c_{-\kb}$, and $\mb=\{m_\kb\}_{\kb\in\Lambda_0}$ such that $m_\kb=m_{-\kb}$, find a \emph{rational} function $\Phi$ in $\Scb$ such that 
\alg{ \label{moment_eqns1} c_\kb &=
\int_{\Tbb^d}e^{i\innerprod{\kb}{\thetab}}\Phi(e^{i\thetab})\d\m(\thetab)\; \; \forall  \kb\in\Lambda,\\
\label{moment_eqns2}m_\kb &=
\int_{\Tbb^d}e^{i\innerprod{\kb}{\thetab}}\log\Phi(e^{i\thetab})\d\m(\thetab)\; \; \forall  \kb\in\Lambda_0,}
where 
$\d\m(\thetab)=\frac{1}{(2\pi)^d}\prod_{j=1}^{d} \d\theta_j$
is the normalized Lebesgue measure on $\Tbb^d$.
\end{problem}

The latter is referred to as a multidimensional rational covariance and cepstral extension problem. In what follows, we always assume that $\cb$ is such that there exists a power spectral density $\Phi$ whose covariance lags indexed in $\Lambda$ coincide with $\cb$; the precise condition will be given in \cref{sec:regularization}, see \cref{assump_feasible}.

A solution to \cref{pb_orig} could be given by setting up the following maximum entropy problem \cite{RKL-16multidimensional}:
\alg{ \label{ME_classic}&\underset{\Phi \in\Scb}{\max}\ \int_{\Tbb^d} \log \Phi \,\d\m \nn\\ 
& \hbox{ s.t. \cref{moment_eqns1} -- \cref{moment_eqns2} hold,}  }
where the variable $\thetab$ or $e^{i\thetab}$ is omitted in the integral when there is no danger of confusion.
Such a problem, however, may not admit a solution in general. Indeed, the dual problem is 
\alg{\label{dual_classic}&\underset{Q, P}{\min\, }\ \langle\cb,\qb\rangle-\langle\mb,\pb\rangle {+} \int_{\Tbb^d} P \log\left( \frac{P}{Q}\right)\d\m \nn\\ 
& \hbox{ s.t. } P\in \overline{\Pfrak}_{+,o},\quad Q\in\overline{\Pfrak}_+, }
where: 
\begin{itemize}
	\item the real variables $\pb=\{p_\kb\}_{\kb\in\Lambda_0}$ and $\qb=\{q_\kb\}_{\kb\in\Lambda}$ are the Lagrange multipliers associated to the constraints \cref{moment_eqns2} and \cref{moment_eqns1}, respectively;
	\item $\langle \cb, \qb\rangle : =\sum_{\kb\in\Lambda} c_\kb q_\kb$ denotes the inner product between two multisequences indexed in $\Lambda$, and in $\innerprod{\mb}{\pb}$ one fixes $p_{\zerob}=1$ and $m_{\zerob}$ arbitrarily;
	\item the Lagrange multipliers inherit the symmetry from the data $\cb$ and $\mb$ in the sense that $p_{-\kb} = p_{\kb}$ and $q_{-\kb}=q_{\kb}$, so that
	\begin{equation}\label{sym_poly_PQ}
	P(e^{i\thetab}):=\sum_{\kb\in\Lambda} p_\kb e^{-i\innerprod{\kb}{\thetab}} \quad \text{and} \quad Q(e^{i\thetab}):=\sum_{\kb\in\Lambda} q_\kb e^{-i\innerprod{\kb}{\thetab}}
	\end{equation}
	are symmetric trigonometric polynomials;
	\item the set $\overline{\Pfrak}_+$ is the closure of $\Pfrak_+$, and the latter is the cone of positive polynomials, namely
	$\Pfrak_+ := \left\{ Q(e^{i\thetab})=\sum_{\kb\in\Lambda} q_\kb e^{-i\innerprod{\kb}{\thetab}} \,:\, Q(e^{i\thetab})>0 \quad \forall\thetab\in\Tbb^d \right\}$;
	\item $\Pfrak_{+,o} := \{\, P\in\Pfrak_+ \,:\, p_{\mathbf 0}=1\, \}$ is the normalized version of $\Pfrak_{+}$, and similarly $\overline{\Pfrak}_{+,o}$ denotes the closure.
\end{itemize}
If the dual problem \cref{dual_classic} admits a solution such that  $(\hat P,\hat Q)\in \Pfrak_{+,o}\times \Pfrak_+$, then it is unique, the duality gap is equal to zero, and the primal solution is rational, i.e., of the form
\alg{\label{form_opt_log} \hat\Phi := {\hat P/\hat Q}\,; }
otherwise, if the dual solution is on the boundary of the feasible set, then  the duality gap is (usually) different from zero and it is not guaranteed that such $\hat \Phi$ solves the primal problem \cref{ME_classic}. A partial remedy is to consider the regularization approach proposed in \cite{enqvist2004aconvex}:
\alg{\label{dual_reg_log}&\underset{Q,P}{\min\, }\ \langle\cb,\qb\rangle-\langle\mb,\pb\rangle {+} \int_{\Tbb^d} P\log\left( \frac{P}{Q}\right)\d\m - \lambda \int_{\Tbb^d} \log P\,\d\m \nn\\ 
& \hbox{ s.t. } P\in \overline{\Pfrak}_{+,o},\quad Q\in\overline{\Pfrak}_+, }
where $\lambda>0$ is a regularization parameter. In the case that $d\leq 2$, its solution $(\hat P,\hat Q)\in \Pfrak_{+,o}\times \Pfrak_+$ is unique so that the corresponding rational function \cref{form_opt_log} matches $\cb$ and approximately matches $\mb$. The smaller $\lambda$ is, the better the approximation is. However, in the case that $d\geq 3$, the solution to \cref{dual_reg_log} may belong to the boundary of the feasible set. If that happens, it is not guaranteed that \cref{form_opt_log} is a solution to the primal problem \cref{ME_classic} in the sense that the covariance matching constraint \cref{moment_eqns1} is likely to be violated.

\section{Generalized problem}\label{sec:generalized}

The aim of this section is to provide a more general formulation of Problem \cref{ME_classic} in such a way that it is possible to guarantee the existence of a rational power spectral density (in the sense that it maximizes the entropy matching $\cb$ and approximately $\mb$) even in the case $d\geq 3$. The fundamental idea is to use a more general notion of the entropy as well as the cepstral coefficient.

\begin{definition}[\cite{tokuda1990generalized}] \label{def_cep_a}
Given the power spectral density $\Phi \in\Scb$ and a real number $\alpha\in [0,1]$, the generalized cepstral coefficients  are defined as 
\begin{equation}\label{def_cep_gam0}
m_{\alpha,\kb} = \begin{cases}
\int_{\Tbb^d} e^{i\innerprod{\kb}{\thetab}} \log \Phi\, \d\mu & \text{if } \alpha=0, \\
\\
\frac{1}{\alpha}   \int_{\Tbb^d} e^{i\innerprod{\kb}{\thetab}} \Phi^\alpha \d\mu   & \text{if } \alpha\in(0,1] \text{ and } \kb\neq \zerob,
\end{cases}
\end{equation}
and $m_{\alpha,\zerob} := \frac{1}{\alpha}  ( \int_{\Tbb^d} \Phi^\alpha \d\mu -1)$ for $\alpha\in(0,1]$.
\end{definition}
In the case that $\alpha=0$, we obtain the usual cepstral coefficients in \cref{pb_orig}. 
To ease the notation, we drop the subscript $\alpha$ and just write $m_\kb$ instead.

At this point we need to consider a suitable definition of the entropy which is ``consistent'' with \cref{def_cep_a}. An entropic index should measure how close the spectral density is to the normalized white noise, i.e., a random field having a spectral density identically equal to one. 
For this purpose, we consider the $\alpha$-divergence between the spectral densities $\Phi,\Psi\in\Sc$ \cite{Z-14rat}: 
   \alg{\label{def_alpha}\Dbb_\alpha(\Phi\|\Psi)=
 \begin{cases}
	\int_{\Tbb^d} \left[ \Psi\log\left(\frac{\Psi}{\Phi}\right)-\Psi+\Phi \right] \d\m & \text{if } \alpha=0, \\
	\\
	\int_{\Tbb^d} \left[ \frac{1}{\alpha(\alpha-1)}\Phi^\alpha \Psi^{1-\alpha} + \frac{1}{1-\alpha}\Phi + \frac{1}{\alpha}\Psi \right] \d\m   & \text{if } \alpha\in\Rbb \setminus\{0,1\},\\ 
	\\
	\int_{\Tbb^d} \left[ \Phi\log\left(\frac{\Phi}{\Psi}\right)-\Phi+\Psi \right] \d\m & \text{if } \alpha=1.
	\end{cases} 
   }
 Then, we define as an entropy of $\Phi$:
	\alg{\label{def_alfa_entro} \widetilde{\Hbb}_\alpha(\Phi)&=-\Dbb_\alpha(\Phi\|{1})+\frac{1}{1-\alpha}\left(\int_{\Tbb^d}\Phi\,\d\mu - 1\right)}
	with $\alpha\in\Rbb\setminus\{1\}$.
In what follows, we consider the parametrization $\alpha=1-\nu^{-1}$ 
with an integer $\nu\geq 1$. Thus, we have
\alg{\label{gen_def}m_\kb&=
\begin{cases}
	 \int_{\Tbb^d}e^{i\innerprod{\kb}{\thetab}}\log \Phi \d\m & \text{if } \nu=1,  \\ \\
\frac{\nu}{\nu-1}\int_{\Tbb^d}e^{i\innerprod{\kb}{\thetab}}\Phi^{\frac{\nu-1}{\nu}} \d\m  & \text{if } \nu>1  \text{ and } \kb\neq \zerob,
	\end{cases} 
 \nn\\ 
  \\ 
\Hbb_\nu(\Phi)&=
\begin{cases}
	   \int_{\Tbb^d}\log \Phi \d\m & \text{if } \nu=1,  \\ \\
	\frac{\nu^2}{\nu-1} \left(\int_{\Tbb^d}\Phi^{\frac{\nu-1}{\nu}}\d\m-1\right) & \text{if } \nu>1,
	\end{cases}  
	 \nn}
which are referred to as the $\nu$-cepstral coefficient and the $\nu$-entropy, respectively. We have omitted $m_{\zerob}$ in the case of $\nu>1$ because it will not be used in the later development.

We are now ready to introduce our generalized extension problem:
\begin{subequations}\label{primal_pb}
\begin{alignat}{2}
& \underset{ \Phi\in\Scb}{\max}
& \quad &{ \Hbb_\nu(\Phi)} \nn \\
&\hspace{0.4cm} \text{s.t.}
& \quad & c_\kb=\int_{\Tbb^d}e^{i\innerprod{\kb}{\thetab}}\Phi \d\m \;\ \forall \kb\in\Lambda, \label{constraint_cov} \\
& 
& \quad & m_\kb= \frac{\nu}{\nu-1}\int_{\Tbb^d}e^{i\innerprod{\kb}{\thetab}}\Phi^{\frac{\nu-1}{\nu}} \d\m\;\ \forall \kb\in\Lambda_0. \label{constraint_cep}
\end{alignat}
\end{subequations}
In the  case $\nu= 1$, Problem \cref{primal_pb}
reduces to Problem \cref{ME_classic} because of \cref{gen_def}. 
In the limit case $\nu= \infty$, Problem \cref{primal_pb} reduces to the following one:
\begin{equation} \label{pb_georgiou}
\begin{aligned}
& \underset{\Phi\in\Scb}{\max}
& & -\int_{\Tbb^d}  \Phi\log \Phi \,\d\m\\
&\hspace{0.6cm} \text{s.t.}
& & c_\kb=\int_{\Tbb^d}e^{i\innerprod{\kb}{\thetab}}\Phi \d\m \;\ \forall \kb\in\Lambda.\\
\end{aligned}
\end{equation} 
Indeed, regarding the objective functional, we have by \cref{def_alfa_entro}
\alg{\label{limit_Hinf}\Hbb_\nu(\Phi)=-\Dbb_{1-\nu^{-1}}(\Phi\|1)+\nu\left(\int_{\Tbb^d}\Phi \d\m-1\right).}
We can equivalently maximize $-\Dbb_{1-\nu^{-1}}(\Phi\|1)$ because the covariance constraint for $\kb=\zerob$ fixes the second term of the right hand side in \cref{limit_Hinf}, i.e., $c_{\zerob}=\int_{\Tbb^d}\Phi \d\m$.
Then, by \cref{def_alpha} we have that
\alg{\left. -\Dbb_{1-\nu^{-1}}(\Phi\|1)\,\right|_{\nu=\infty}=-\Dbb_{1}(\Phi\|1)=-\int_{\Tbb^d} \left( \Phi\log \Phi -\Phi+1\right) \,\d\m.}
In view of the zeroth moment $c_{\zerob}$, we can simply maximize the so called {\em von Neumann} entropy: $-\int_{\Tbb^d}  \Phi\log \Phi \,\d\m$.
Regarding the cepstral constraint, we have $m_\kb=\int_{\Tbb^d}e^{i\innerprod{\kb}{\thetab}}\Phi \d\m$ for $\kb\in\Lambda_0$ and $\nu=\infty$,
i.e., we just exploited the definition in \cref{def_cep_gam0} with $\alpha=1$; thus, the {generalized} cepstral constraints become the covariance constraints.

It is important to notice that Problem \cref{pb_georgiou} admits a solution (although it is not rational, see \cite{Georgiou-06}) which is possible to characterize through the (finite dimensional) dual optimization problem. We conclude that Problem \cref{primal_pb} allows to range from a problem which not necessarily admits a solution for $d\geq 3$ (case $\nu=1$) to a problem which admits a solution for any $d$ (limit case $\nu= \infty$). Accordingly, the hope is that for $d\geq 3$ fixed there exists one $\nu$ sufficiently large  such that Problem \cref{primal_pb} admits a rational solution (in the sense that it maximizes the $\nu$-entropy matching $\cb$ and approximately $\mb$). The next two sections of the paper are devoted to proving {that this is indeed the case}. 

\begin{remark}
The choice of the functional space for the spectral density function $\Phi$ in Problem \cref{primal_pb} deserves a few lines to explain. In order to guarantee the integrability of both $\Phi$ and $\Phi^{\frac{\nu-1}{\nu}}$ for $\nu>1$, where the latter one appears in the definitions of the objective functional and the $\nu$-cepstral coefficient, one should take $\Scb$, the set of nonnegative functions, as a subset of the intersection
\begin{equation}
L^1(\Tbb^d)\cap L^{\frac{\nu-1}{\nu}}(\Tbb^d).
\end{equation}
However, by a result on the inclusion relation between $L^p$ spaces \cite{Villani-85}, we have $L^{1}(\Tbb^d)\subset L^{\frac{\nu-1}{\nu}}(\Tbb^d)$, because $\Tbb^d$ has a finite Lebesgue measure and the fraction $\frac{\nu-1}{\nu}<1$. Hence the proper choice of the functional space for Problem \cref{primal_pb}, the previous intersection, reduces again to $L^1(\Tbb^d)$.
\end{remark}

\section{Dual analysis} 
\label{sec:dual}

The case $\nu=1$ has already been analyzed in \cite{RKL-16multidimensional}.
Assume now that $\nu \geq 2$ is fixed. Then, we can equivalently consider the objective functional in \cref{primal_pb} multiplied by $\nu^{-1}$ and discard the constant term not depending on $\Phi$. Thus, the Lagrangian is 
\begin{equation}\label{Lagrangian}
\begin{split}
\Lcal_\nu(\Phi,P,Q) & =\frac{\nu}{\nu-1}\int_{\Tbb^d}  \Phi^{\frac{\nu-1}{\nu}} \d \m - \sum_{\kb\in\Lambda}  \left[ q_\kb \left(\int_{\Tbb^d} e^{i\innerprod{\kb}{\thetab}} \Phi \d \m - c_{\kb}\right)^* \right] \\
 & \quad +\sum_{\kb\in\Lambda_0}  \left[ p_\kb \left(\frac{\nu}{\nu-1}\int_{\Tbb^d} e^{i\innerprod{\kb}{\thetab}} \Phi^{\frac{\nu-1}{\nu}}\d \m - m_{\kb}\right)^* \right] \\ 
 &=\frac{\nu}{\nu-1}\int_{\Tbb^d}  P \Phi^{\frac{\nu-1}{\nu}} \d \m -    \int_{\Tbb^d} Q \Phi \d \m + \langle \qb, \cb\rangle  -\langle \pb, \mb\rangle,
\end{split}
\end{equation}
where $\pb=\{p_\kb\}_{\kb\in\Lambda}$ and $\qb=\{ q_\kb\}_{\kb\in\Lambda}$ are the Lagrange multipliers such that $p_\zerob=1$, $p_{\kb}=p_{-\kb}$ and  $q_{\kb}=q_{-\kb}$, similar to the treatment after the classic formulation \cref{dual_classic}. Moreover, $P$ and $Q$
are symmetric trigonometric polynomials as given in \cref{sym_poly_PQ}. Once again we set $m_\zerob$ equal to an arbitrary but fixed number in the inner product $\innerprod{\pb}{\mb}$.

Next we fix the Lagrange multipliers $\qb$ and $\pb$, or equivalently the polynomials $Q$ and $P$, and consider the problem
\begin{equation}\label{problem_sup_Lagrangian}
\sup_{\Phi\in\Scb} \ \Lcal_\nu(\Phi,P,Q).
\end{equation}
The last two inner products in \cref{Lagrangian} do not depend on $\Phi$, and thus can be ignored in the $\sup$ problem.
We first claim that in order for the supremum to be finite, it is necessary that $Q\in\overline{\Pfrak}_{+}$. To see this, consider the following two cases.

	{\em First case.} If $Q(e^{i\phib})<0$ for some $\phib\in\Tbb^d$ and $P(e^{i\phib}) > 0$, then due to continuity, there exists a (sufficiently small) ball $B$ in $\Tbb^d$ centered in $\phib$ such that $Q(e^{i\thetab})<0$ and $P(e^{i\thetab})>0$ for all $\thetab\in B$.	
	We can take 
	\begin{equation}
	\Phi(e^{i\thetab}) = \begin{cases}
	w|Q(e^{i\thetab})| & \text{for } \thetab\in B, \\
	0 & \text{elsewhere},
	\end{cases}
	\end{equation}
	where $w$ is a positive real number. Letting $w\to\infty$, such a choice of $\Phi$ leads to 
	\begin{equation}\label{L_Phi_special}
	\begin{split}
	\Lcal_\nu(\Phi,P,Q) & = \frac{\nu}{\nu-1} \int_{B} P \Phi^{\frac{\nu-1}{\nu}}\d\m - \int_{B} Q\Phi \d\m +\const  \\
	 & > \int_{B} w|Q|^2 \d\m +\const	 \to \infty,
	\end{split}
	\end{equation}
	where the inequality holds because the integrand $P \Phi^{\frac{\nu-1}{\nu}}$ is positive on $B$ and the fraction $\frac{\nu}{\nu-1}>0$ (since we have taken the integer $\nu\geq2$).
	
	{\em Second case}. If $Q(e^{i\phib})<0$ for some $\phib\in\Tbb^d$ and $P(e^{i\phib}) \leq 0$, we can take the same $\Phi$ as in the previous {case}, and it turns out that the integral terms in \cref{L_Phi_special}, namely
	\begin{equation}\label{int_in_L}
    \int_{B} \left(\frac{\nu}{\nu-1} P \Phi^{\frac{\nu-1}{\nu}} - Q\Phi \right)\d\m
	\end{equation}
	 also tend to infinity as $w\to\infty$. To show this, we only need to compare the two functions inside the integral:
	\begin{equation}
	\begin{split}
	\left|\frac{\frac{\nu}{\nu-1} P \Phi^{\frac{\nu-1}{\nu}}}{Q \Phi} \right| & = \frac{\nu}{\nu-1} \left|\frac{P}{Q}\right| \Phi^{-\frac{1}{\nu}} = \frac{\nu}{\nu-1} \left|\frac{P}{Q}\right| (w|Q|)^{-\frac{1}{\nu}} \\
	 & \leq \frac{\nu P_{\max}}{(\nu-1)Q_{\min}^{1+\frac{1}{\nu}}} w^{-\frac{1}{\nu}} \to 0 \quad \text{as} \quad w\to\infty,
	\end{split}
	\end{equation}
	where $P_{\max}:=\max_{\thetab\in\overline{B}} |P(e^{i\thetab})|$ and $0<Q_{\min}:=\min_{\thetab\in\overline{B}} |Q(e^{i\thetab})|$ are constants.
	It is clear that the above limit holds uniformly for all $\thetab\in B$. Therefore, the integral \cref{int_in_L} is dominated by the term $\int_{B}(-Q\Phi)\d\m$ which, as we have shown in the previous case, diverges as $w\to\infty$. We conclude again that the supremum of $\Lcal_\nu(\Phi,P,Q)$ in this case is $\infty$.

Now, we assume $Q\in\overline{\Pfrak}_+$ and discuss the role played by the polynomial $P$. Consider the set
\begin{equation}
B_1 := \left\{ \thetab\in\Tbb^d : P(e^{i\thetab}) < 0 \right\}
\end{equation}
which is measurable by construction. Then we can split the domain of integration and write the Lagrangian as
\begin{equation}
\begin{split}
\Lcal_\nu(\Phi,P,Q) & = \left(\int_{B_1}  + \int_{\Tbb^d\backslash B_1} \right) \left(\frac{\nu}{\nu-1} P \Phi^{\frac{\nu-1}{\nu}} - Q\Phi \right)\d\m + \const \\
 & \leq \int_{\Tbb^d\backslash B_1} \left(\frac{\nu}{\nu-1} P \Phi^{\frac{\nu-1}{\nu}} - Q\Phi \right)\d\m + \const
\end{split}
\end{equation}
where the inequality holds because the integrand is nonpositive on $B_1$. It follows that if $\Phi$ stays positive on a subset of $B_1$ with a positive Lebesgue measure, then it cannot achieve the supremum of the Lagragian since setting $\Phi|_{B_1}=0$ a.e. increases the value of $\Lcal_\nu(\Phi,P,Q)$. Therefore, the sup problem in \cref{problem_sup_Lagrangian} is equivalent to the following one:
\begin{equation}\label{modified_Lagrangian}
\sup_{\Phi\in \Scb} \ \Lcal'_\nu(\Phi,P,Q) := \int_{\Tbb^d\backslash B_1} \left(\frac{\nu}{\nu-1} P \Phi^{\frac{\nu-1}{\nu}} - Q\Phi \right)\d\m,
\end{equation}
where $\Lcal'_\nu$ is called the modified Lagrangian. In order to exclude pathological cases, we further remove points that are in the zero sets of the polynomials $P$ and $Q$ from the domain of integration. More precisely, define the zero set of a trigonometric polynomial $P$ as
\begin{equation}\label{zero_set_P}
\Zcal(P) := \left\{ \thetab\in\Tbb^d : P(e^{i\thetab})=0 \right\}
\end{equation}
and let $\Zcal:=\Zcal(P)\cup \Zcal(Q)$. It is then well known that $\mu(\Zcal)=0$, and the modified Lagrangian 
\begin{equation}
\Lcal'_\nu(\Phi,P,Q) = \int_{S} \left(\frac{\nu}{\nu-1} P \Phi^{\frac{\nu-1}{\nu}} - Q\Phi \right)\d\m
\end{equation}
where $S:=\Tbb^d\backslash B_1\backslash \Zcal$. Apparently, the value of $\Phi$ on $\Zcal$ can be set arbitrarily since that part does not contribute to the value of the Lagrangian.

For a feasible direction $\delta\Phi\in L^1(\Tbb^d)$ such that $\Phi+\varepsilon\delta\Phi\geq0$ a.e. on $\Tbb^d$ for $\varepsilon>0$ sufficiently small, compute the directional derivative
\alg{\label{first_variation_L}
	\delta\Lcal'_\nu(\Phi,P,Q; \delta \Phi) = \int_{S}  \left( P \Phi^{-\frac{1}{\nu}} - Q \right) \delta\Phi \d \m,}
where we have interchanged the differential and integral operators.
Then we impose the equality a.e. on $S$
\begin{equation}
P \Phi^{-\frac{1}{\nu}} - Q = 0  \implies \Phi = \left({P}/{Q}\right)^\nu.
\end{equation}
The latter $\Phi$ is a positive rational function well-defined on $S$ where both $P$ and $Q$ are positive. Moreover, such $\Phi$ makes the directional derivative \cref{first_variation_L} vanish in any feasible direction, and hence is a stationary point of $\Lcal_\nu'(\Phi,P,Q)$.
It is not difficult to show that $\Lcal_\nu'(\Phi,P,Q)$ is a concave functional in $\Phi$ (via pointwise reasoning with the integrand following the idea of \cite[Lemma 4.1]{zhu2020m}) so that a stationary point is indeed a maximizer. To summarize, the supremum of the Lagrangian is attained at the function
\begin{equation}\label{Phi_nu_unreg}
\Phi_\nu = 
\begin{cases}
\left({P}/{Q}\right)^\nu & \text{for } \thetab\in S, \\
0 & \text{elsewhere}. \\
\end{cases}
\end{equation}
Plugging the above optimal form $\Phi_\nu$ into the original Lagrangian \cref{Lagrangian}, we obtain the dual function
\begin{equation}\label{dual_func}
J_{\nu}(P,Q) := \frac{1}{\nu-1}\int_{S}  \frac{P^\nu}{Q^{\nu-1}} \d \m +\langle \qb, \cb\rangle  -\langle \pb,\mb\rangle,
\end{equation}
and the dual optimization problem is 
\alg{\label{dual_pb} 
	\underset{P,Q}{\min\,}\ J_{\nu}(P,Q) \ \hbox{ s.t. } p_{\zerob}=1, \quad Q\in\overline{\Pfrak}_+.
}

\begin{remark}\label{rem_J_nu_convex}
From the Lagrange duality theory (see e.g., \cite{bv_cvxbook}), we know that the dual problem \cref{dual_pb} is convex. However, it is not obvious at all why the objective function $J_{\nu}$ in \cref{dual_func} is convex, especially since the set $S$ depends on the polynomial $P$. The following argument gives an alternative route to convexity which works for an \emph{odd} $\nu\geq 2$. Consider the composite (piecewise) function defined on the upper half plane $\{(x,y)\in\Rbb^2 : y>0\}$:
\begin{equation}\label{func_g}
g(x,y):=\max\{0, x^\nu/y^{\nu-1}\} = \begin{cases}
x^\nu/y^{\nu-1} & \text{ for }\ x>0,  \\
0 & \text{ for }\ x\leq 0.
\end{cases}
\end{equation}
According to \cref{sec:appendix_A}, the function $g$ is convex on its domain of definition (which is not entirely trivial). Then the integral term in \cref{dual_func} can be rewritten as
\begin{equation}
\int_{S}  \frac{P^\nu}{Q^{\nu-1}} \d \m = \int_{\Tbb^d\backslash \Zcal(Q)} g(P(e^{i\thetab}), Q(e^{i\thetab}))\, \d\m.
\end{equation}
The convexity of this term can be proved via pointwise reasoning with the integrand in the style of Proposition~5.3 in \cite{zhu2020m}, since the integral operation is linear. The fact that $J_\nu$ is convex then follows easily.
\end{remark}

In the special case with $P\in \Pfrak_{+,o}$ and $Q\in \Pfrak_+$, the optimal form $\Phi_\nu$ in \cref{Phi_nu_unreg} becomes simply
\alg{\label{Phi_nu}\Phi_\nu = \left({P}/{Q}\right)^\nu, }
and we have the following result.

\begin{proposition} If Problem \cref{dual_pb} admits a solution $(\hat P,\hat Q)\in \Pfrak_{+,o} \times \Pfrak_+$, then the rational spectrum $\hat \Phi_\nu=( {\hat P}/{\hat Q} )^\nu$ solves Problem \cref{primal_pb}.
\end{proposition}
\begin{proof}  
First, for $(P,Q)\in  {\Pfrak}_{+,o} \times{\Pfrak}_+ $ the dual function $J_\nu$ takes the same form as \cref{dual_func},
and the only difference is that the domain of integration becomes the whole $\Tbb^d$.
Clearly, if $(\hat P,\hat Q)\in \Pfrak_{+,o} \times\Pfrak_+$, i.e., the solution of the dual problem lies in the interior of the feasible set,  then such a solution must be a stationary point of the convex function $J_\nu$.
Next, one can check the stationarity conditions of $J_\nu$ with respect to $P$ and $Q$ (which are similar to those in the proof of \cref{thm_main} and hence omitted here) 
in order to see that the rational function $\hat \Phi_\nu=( {\hat P}/{\hat Q} )^\nu$ satisfies the constraints in \cref{primal_pb}. Finally, the duality gap is zero, and thus the claim is proved.
\end{proof}

However, the difficult part here is that in general, it is not obvious at all why the dual problem \cref{dual_pb} should have a solution in the positive region $\Pfrak_{+,o} \times \Pfrak_+$. In order to promote such a positive solution, we shall employ a suitable \emph{regularization} technique in the next sections. As we will see, the price to pay for including the positivity constraint is that, the exact $\nu$-cepstral matching cannot be achieved.

\section{Regularizing the dual problem}\label{sec:regularization}

In the case that $\nu=1$, the remedy to guarantee the existence of a solution to the dual problem in the interior is to add a penalty term for $P$ in the dual function, see \cref{dual_reg_log}. As pointed out in \cite{enqvist2004aconvex}, the regularizer is the opposite of the entropy of $P$. In particular, if the regularization parameter $\lambda$ approaches infinity, then $P$ tends to minimize the penalty term whose minimum is attained for $P=1$ corresponding to the classical \emph{Maximum Entropy} solution {\cite{burg1975maximum}}. 
Such a penalty term, however, admits another interpretation. Indeed, we have
$- \int_{\Tbb^d}\log P \d\m =  \int_{\Tbb^d}\log\left({1}/{P}\right)\d\m =  \Hbb_1\left({1}/{P}\right)$, 
where $ \Hbb_1(\cdot) := \widetilde\Hbb_0(\cdot)$ is the usual entropy in \cref{ME_classic}.
Accordingly, the regularizer in \cref{dual_reg_log} can be also understood as the entropy of the inverse of the numerator of $\Phi=P/Q$.

Following the latter philosophy, in our case the numerator of $\Phi_\nu$ in \cref{Phi_nu} is $P^\nu$ and thus we consider as a regularizer the $\nu$-entropy of $1/P^\nu$:
\alg{
	 {\Hbb}_\nu\left(\frac{1}{P^\nu}\right) = \frac{\nu^2}{\nu-1}\int_{\Tbb^d} \left(\frac{1}{P^\nu}\right)^{\frac{\nu-1}{\nu}}\d\m = \frac{\nu^2}{\nu-1}\int_{\Tbb^d} \frac{1}{P^{\nu-1}}\d\m.
}
Thus, the regularized dual problem is
\begin{equation}
\begin{aligned}\label{reg_dual_pb}
	& \underset{P,Q}{\min\,}\ J_{\nu,\lambda}(P,Q) := J_{\nu}(P,Q) + \frac{\lambda}{\nu-1} \int_{\Tbb^d}\frac{1}{P^{\nu-1}}\d\m \\
	& \hbox{ s.t. } P\in\overline{\Pfrak}_{+,o},\quad Q\in\overline{\Pfrak}_+,
\end{aligned}
\end{equation}
where $\lambda>0$ is the regularization parameter which has absorbed the positive constant $\nu^2$. Notice that $P$ is required to belong to $\overline\Pfrak_{+,o}$ in \cref{reg_dual_pb} because the $\nu$-entropy is defined only for nonnegative functions. Next we need the following feasibility assumption.

\begin{assumption}[Feasibility]\label{assump_feasible}
	The given covariances $\{c_\kb\}_{\kb\in\Lambda}$ admit an integral representation
	$c_\kb = \int_{\Tbb^d} e^{i\innerprod{\kb}{\thetab}} \Phi_0 \d\m$ for all $\kb\in\Lambda$, where the function $\Phi_0 \in\Scb$	is strictly positive on some open ball $B_2\subset\Tbb^d$.
\end{assumption}

We are now ready to state the main result of this paper.

\begin{theorem}\label{thm_main}
	Under \cref{assump_feasible}, if $\nu\geq \frac{d}{2}+1$, then the regularized dual problem \cref{reg_dual_pb} admits a unique solution $(\hat P,\hat Q)\in \Pfrak_{+,o}\times \Pfrak_+$.
	Moreover, the positive rational function $\hat{\Phi}_\nu=(\hat{P}/\hat{Q})^\nu$ matches the given covariances $\{c_\kb\}_{\kb \in\Lambda}$ exactly and the $\nu$-cepstral coefficients $\{m_\kb\}_{\kb \in\Lambda_{{0}}}$ approximately.
\end{theorem}

The remaining part of this section is devoted to the proof of the above theorem. To this end, we need some ancillary lemmas.

\begin{lemma}\label{lem_convex}
The regularized dual function $J_{\nu,\lambda}$ is strictly convex on $\Pfrak_{+,o}\times \Pfrak_+$.
\end{lemma}

\begin{proof}
	
	Since we are only asserting strict convexity in the interior of the feasible set $\overline{\Pfrak}_{+,o}\times\overline{\Pfrak}_+$, it suffices to carry out the derivative test for convexity. In order to ease the notation, let us introduce the linear operator
	\begin{equation}\label{Gamma_operator}
	\Gamma: \Phi \mapsto \left\{\int_{\Tbb^d} e^{i\innerprod{\kb}{\thetab}} \Phi \d \m\right\}_{\kb\in\Lambda}
	\end{equation}
	that sends a function $\Phi\in L^1(\Tbb^d)$ to its Fourier coefficients indexed in $\Lambda$.
	One can define a similar linear operator $\Gamma_0$  such that the index set $\Lambda$ in \cref{Gamma_operator} is replaced by $\Lambda_0$.
	After some computation, we get the first-order differentials of the regularized dual function:
	\begin{equation}\label{reg_J_diff}
	\begin{aligned}
	\delta J_{\nu,\lambda}(P,Q;\delta Q) & = \innerprod{\delta\qb}{\cb-\Gamma\left(\left(\tfrac{P}{Q}\right)^\nu\right)} , \\
	\delta J_{\nu,\lambda}(P,Q;\delta P) & = \innerprod{\delta\pb}{\Gamma_0\left(\tfrac{\nu}{\nu-1}\left(\tfrac{P}{Q}\right)^{\nu-1}\right) - \Gamma_0\left(\tfrac{\lambda}{P^\nu}\right) -\mb}.
	\end{aligned}
	\end{equation}
	Notice that due to the strict positivity of $P$ and $Q$, the directions $\delta P$ and $\delta Q$ can be arbitrary in a finite-dimensional vector space, and the operation of interchanging the differential and the integral can be well justified using Lebesgue's dominated convergence theorem.
	Let us continue to compute the \emph{joint} second-order differential with respect to 
	the directions $(\delta P^{(1)}, \delta Q^{(1)})$ and $(\delta P^{(2)}, \delta Q^{(2)})$:
		
	\begin{equation}\label{reg_J_sec_diff}
	\begin{aligned}
	& \delta^2 J_{\nu,\lambda} (P,Q; \delta P^{(1)}, \delta Q^{(1)}, \delta P^{(2)}, \delta Q^{(2)}) \\
	= & \nu \int_{\Tbb^d} \left[ \frac{P^{\nu-2}}{Q^{\nu+1}} (P\delta Q^{(1)} - Q\delta P^{(1)}) (P\delta Q^{(2)} - Q\delta P^{(2)}) + \frac{\lambda \delta P^{(1)} \delta P^{(2)}}{P^{\nu+1}} \right] \d\m
	\end{aligned}
	\end{equation} 
	which is a \emph{bilinear form} in $(\delta P^{(1)}, \delta Q^{(1)})$ and $(\delta P^{(2)}, \delta Q^{(2)})$. 
	Taking $\delta P^{(1)}=\delta P^{(2)}=\delta P$ and $\delta Q^{(1)}=\delta Q^{(2)}=\delta Q$, we obtain the second-order \emph{variation}:
	\begin{equation}\label{reg_J_sec_vari}
	\begin{split}
	\delta^2 J_{\nu,\lambda} (P,Q;\delta P, \delta Q) = \nu \int_{\Tbb^d} \left[ \frac{P^{\nu-2}}{Q^{\nu+1}} (P\delta Q - Q\delta P)^2 + \frac{\lambda \delta P^2}{P^{\nu+1}} \right] \d\m \geq 0.
	\end{split}
	\end{equation}
	The last inequality holds because $\nu\geq 2$ and the integrand in the bracket is nonnegative. Moreover, if we impose $\delta^2 J_{\nu,\lambda} (P,Q;\delta P, \delta Q) = 0$, then such an integrand has to vanish identically due to continuity. It follows easily that $\delta P = \delta Q \equiv 0$ which implies that they are zero polynomials (cf. e.g., \cite[Lemma~1]{ringh2015multidimensional}). We can now conclude that the second-order differential of $J_{\nu,\lambda}$ is positive definite, which completes the proof.
\end{proof}

\begin{remark}
The above lemma can be strengthened in the sense that the strict convexity of $J_{\nu,\lambda}$ can be extended to the boundary of the feasible set via similar reasoning to that in \cref{rem_J_nu_convex}. Here $J_{\nu,\lambda}$ should be understood as an extended real-valued function. To this end, fix an integer $\nu\geq 2$ and a real parameter $\lambda>0$, define the bivariate function $g_2(x,y) := x^\nu/y^{\nu-1} + \lambda/x^{\nu-1}$ for $x>0$ and $y>0$.
Then it is easy to show via the derivative test (cf.~\cref{derivatives_g_1} in \cref{sec:appendix_A}) that $g_2$ is strictly convex in its domain of definition.
Moreover, we can rewrite the regularized dual function as
\begin{equation}
J_{\nu,\lambda}(P,Q) = \frac{1}{\nu-1} \int_{\Tbb^d\backslash \Zcal} g_2(P(e^{i\thetab}), Q(e^{i\thetab})) \d\mu + \innerprod{\qb}{\cb} - \innerprod{\pb}{\mb},
\end{equation}
where $\Zcal=\Zcal(P)\cup \Zcal(Q)$ has no effect on the value of the integral.
A pointwise argument with the integrand yields the strict convexity of $J_{\nu,\lambda}$ by definition. The difference now is that $(P,Q)$ is not restricted to the interior of the feasible set.
\end{remark}

\begin{lemma}\label{lem_lower_semicont}
	The function $J_{\nu,\lambda}$ is lower-semicontinuous on $\overline{\Pfrak}_{+,o}\times \overline{\Pfrak}_+$.
\end{lemma}

\begin{proof}
	As shown in the previous lemma, the function $J_{\nu,\lambda}$ is differentiable (in fact, \emph{smooth}) in the interior of the domain, and \emph{a fortiori} continuous. Hence we only need to prove the lower-semicontinuity on the boundary of the domain. In view of the definitions of $J_\nu$ and $J_{\nu,\lambda}$ in \cref{dual_func} and \cref{reg_dual_pb}, it remains to work on the integral terms $\int_{\Tbb^d}\frac{P^\nu}{Q^{\nu-1}}\d\m$ and $\int_{\Tbb^d} \frac{1}{P^{\nu-1}}\d\m$
	since the inner products are obviously continuous. In the following, we only prove the lower-semicontinuity of the term $\int_{\Tbb^d}\frac{P^\nu}{Q^{\nu-1}}\d\m$ as the proof holds almost verbatim for the other term.
	
	For a sequence $\{(P_j, Q_j)\}_{j\geq 1}$ that tends to $(P,Q)$ on the boundary of $\overline{\Pfrak}_{+,o}\times \overline{\Pfrak}_+$, there are two cases.
	
{\em First case.} $Q\in\partial\Pfrak_{+}$ where the symbol $\partial$ denotes the boundary of a set.
		Notice that the union of the zero sets $\tilde{\Zcal}:=\bigcup_{j=1}^\infty \Zcal(Q_j) \cup \Zcal(Q)$ (which again has a Lebesgue measure zero) can be safely excluded from the domain of integration so that all the integrands are well defined. Moreover, we have $P_j^\nu/Q_j^{\nu-1}\to P^\nu/Q^{\nu-1}$ pointwise as $j\to\infty$. By Fatou's lemma \cite[p.~23]{rudin1987real}, the inequality
		\begin{equation}
		\int_{\Tbb^d\backslash\tilde{\Zcal}} \frac{P^\nu}{Q^{\nu-1}} \d\m \leq \liminf_{j\to\infty} \int_{\Tbb^d\backslash\tilde{\Zcal}} \frac{P_j^\nu}{Q_j^{\nu-1}} \d\m
		\end{equation}
		holds, which means that the integral $\int_{\Tbb^d\backslash\tilde{\Zcal}}\frac{P^\nu}{Q^{\nu-1}}\d\m$ is lower-semicontinuous at $(P,Q)$.
		
		{\em Second case.} $Q\in\Pfrak_{+}$ so that $P\in\partial\Pfrak_{+,o}$ must happen. The proof here is almost identical to that in the previous case, and thus it is omitted.

	With the lower-semicontinuity on the boundary proved, the assertion of the lemma follows.	
\end{proof}

The next lemma concerns the limit behavior of the regularized dual function as the norm of its argument goes to infinity, whose proof relies heavily on \cref{assump_feasible}. Let us define the norms of the Lagrange multipliers (dual variables) as follows:
\begin{equation}\label{P_Q_norms}
\|P\|:=\|\pb\| = \sqrt{\sum_{\kb\in\Lambda_0} |p_\kb|^2},\  \|Q\|:=\|\qb\| = \sqrt{\sum_{\kb\in\Lambda} |q_\kb|^2},
\end{equation}
and $\|(P,Q)\| := \|P\| + \|Q\|$.

\begin{lemma}\label{lem_unbounded}
	Suppose that \cref{assump_feasible} holds. If a sequence $\{(P_j,Q_j)\}_{j\geq1}\subset\overline{\Pfrak}_{+,o}\times \overline{\Pfrak}_+$ is such that $\|(P_j,Q_j)\|\to\infty$ as $j\to\infty$, then $J_{\nu,\lambda}(P_j,Q_j)\to\infty$.	
\end{lemma}

\begin{proof}
	By a variant of \cite[Lemma 10]{ringh2015multidimensional}, we know that for a (Laurent) polynomial $P$ nonnegative on $\Tbb^d$, it holds that $|p_\kb|\leq p_\zerob=1$ for any index $\kb\in\Lambda_0$. Therefore, we have $\|P_j\|\leq \sqrt{|\Lambda|-1}$ for all $j$, and it must happen that $\|Q_j\|\to\infty$ in order for the combined norm of $(P_j,Q_j)$ to be unbounded.
	Observe that the integral terms in $J_{\nu,\lambda}$ are nonnegative. Hence we have
	\begin{equation}\label{J_nu_lambda_p1}
	J_{\nu,\lambda}(P_j,Q_j) \geq \innerprod{\qb_j}{\cb} -\innerprod{\pb_j}{\mb} .
	\end{equation}
	The inner product $\innerprod{\pb_j}{\mb}$ still remains bounded. So it suffices to show that $\innerprod{\qb_j}{\cb}$ blows up. It is at this point that \cref{assump_feasible} plays an important role.	
	If our covariance data satisfy the assumption, we can write 	
	\begin{equation}\label{inner_prod_q_c}
	\innerprod{\qb_j}{\cb} = \sum_{\kb\in\Lambda} q_{j,\kb} \overline{\int_{\Tbb^d} e^{i\innerprod{\kb}{\thetab}} \Phi_0 \d\m} = \int_{\Tbb^d} Q_j \Phi_0 \d\m \geq 0
	\end{equation}
	for some nonnegative function $\Phi_0$ on $\Tbb^d$.
	Define the normalized sequence $\qb^0_j := \qb_j/\|\qb_j\|$ so that $\|\qb^0_j\|=1$. It follows that $\innerprod{\qb_j^0}{\cb}=\frac{1}{\|\qb_j\|}\innerprod{\qb_j}{\cb}\geq0$, and we can define the real quantity $\eta:=\liminf_{j\to\infty} \innerprod{\qb_j^0}{\cb}\geq0$. Our target is to show that $\eta>0$. To this end, notice first that the limit inferior implies that there exists a subsequence $\{\qb^0_{j_\ell}\}_{\ell\geq1}$ such that $\eta=\lim_{\ell\to\infty} \innerprod{\qb_{j_\ell}^0}{\cb}$. Since this subsequence is bounded in norm, it has a convergent subsequence which (with an abuse of notation) we still denote with $\{\qb^0_{j_\ell}\}_{\ell\geq1}$. Define the limit $\qb^0_\infty:=\lim_{\ell\to\infty} \qb^0_{j_\ell}$. Then due to the continuity of the inner product, we have $\eta=\innerprod{\qb^0_\infty}{\cb}$. 
	
	Notice on the other hand, that the polynomials $Q^0_j(e^{i\thetab})\geq0$ for all $\thetab\in\Tbb^d$. Due to the convergence of the coefficients $\qb^0_{j_\ell}$, the polynomials $Q^0_{j_\ell}$ converges uniformly to $Q^0_\infty$ for any point in $\Tbb^d$ as $\ell\to\infty$. Thus we have $Q^0_\infty(e^{i\thetab})\geq0$ for any $\thetab\in\Tbb^d$. If $\eta=0$, then similar to \cref{inner_prod_q_c}, we have
	\begin{equation}
	0 = \eta = \innerprod{\qb_\infty^0}{\cb} = \int_{\Tbb^d} Q_\infty^0 \Phi_0 \d\m,
	\end{equation}
	which, given \cref{assump_feasible}, means that $Q^0_{\infty}(e^{i\thetab})\equiv0$ on the open ball $B_2\subset\Tbb^d$ where $\Phi_0$ is positive. Appealing to \cite[Lemma 1]{ringh2015multidimensional} again, this implies that $Q^0_\infty$ is a zero polynomial, a contradiction to the fact that $\qb^0_\infty$ is of unit norm. Therefore, we can conclude that $\eta>0$. Now we can continue \cref{J_nu_lambda_p1}:
	\begin{equation}\label{J_nu_lambda_p2}
	J_{\nu,\lambda}(P_j,Q_j) \geq \innerprod{\qb_j}{\cb} + \text{const.} = \|\qb_j\| \innerprod{\qb_j^0}{\cb} + \text{const.} > \frac{\eta}{2} \|\qb_j\| + \text{const.} \to \infty,
	\end{equation}
	where the last inequality holds by the definition of $\liminf$ for $j$ sufficiently large.
\end{proof}

	The last lemma in this section describes the behavior of $J_{\nu,\lambda}$ on the boundary of the feasible set. The proof hinges on Proposition A.4 in \cite{zhu2020m} whose statement is reported below since it will be referred to several times in the subsequent development.
	
\begin{proposition}[Proposition A.4, \cite{zhu2020m}]\label{prop_pol}
	Let $p:\Rbb^n\to\Rbb$ be a polynomial with $p(\thetab_0)=0$ and assume that $p$ is nonnegative in a $\thetab_0$-centered ball $B_\varepsilon(\thetab_0)$ for some radius $\varepsilon>0$. Then if
	$m\ge n/2$, we have that $\int_{B_\varepsilon(\thetab_0)}p^{-m}\,\d\m = \infty$.
\end{proposition}

\begin{lemma}\label{lem_boundary}
	If $\nu\geq \frac{d}{2}+1$, then $J_{\nu,\lambda}(P,Q)=\infty$ for any $(P,Q)$ on the boundary of the feasible set $\overline{\Pfrak}_{+,o}\times \overline{\Pfrak}_+$.
\end{lemma}

\begin{proof}
	Similar to the proof of \cref{lem_lower_semicont}, we consider two cases.
	\begin{enumerate}
		\item $P\in\partial\Pfrak_{+,o}$. In this case, we have
		\begin{equation}
		J_{\nu,\lambda}(P,Q) \geq \innerprod{\qb}{\cb} - \innerprod{\pb}{\mb} + \frac{\lambda}{\nu-1} \int_{\Tbb^d}\frac{1}{P^{\nu-1}} \d\m,
		\end{equation}
		where the inequality holds because the integral term in $J_\nu$ is nonnegative. The latter integral diverges when $\nu-1\geq d/2$ by \cref{prop_pol} while the inner products take finite values. Hence we have $J_{\nu,\lambda}(P,Q)=\infty$.
		\item $P\in\Pfrak_{+,o}$ and $Q\in\partial\Pfrak_{+}$. Due to continuity, we have $P(e^{i\thetab})\geq P_{\min}:=\min_{\thetab\in\Tbb^d} P(e^{i\thetab})>0$.
		 Now we have
		\begin{equation}
		J_{\nu,\lambda}(P,Q) \geq J_\nu(P,Q) \geq \frac{P_{\min}^\nu}{\nu-1} \int_{\Tbb^d} \frac{1}{Q^{\nu-1}} \d\m + \innerprod{\qb}{\cb} - \innerprod{\pb}{\mb},
		\end{equation}
		where the first inequality comes from the fact that the regularization term is nonnegative. Appealing to \cref{prop_pol} again, we can conclude that the integral of $1/Q^{\nu-1}$ diverges when $\nu-1\geq d/2$. This completes the proof of the lemma.
	\end{enumerate}
\end{proof}

We are now ready to prove \cref{thm_main}.

\begin{proof}[Proof of \cref{thm_main}]
	
	For a sufficiently large real number $\beta$, consider the (nonempty) sublevel set of the regularized dual function
	\begin{equation}\label{sublevel_set_reg}
	J_{\nu,\lambda}^{-1}(-\infty,\beta] := \{ (P,Q)\in \overline{\Pfrak}_{+,o}\times \overline{\Pfrak}_+ \,:\, J_{\nu,\lambda}(P,Q)\leq\beta \}.
	\end{equation}
	We know that
	\begin{itemize}
		\item the sublevel set is closed by \cref{lem_lower_semicont};
		\item the sublevel set is bounded by \cref{lem_unbounded}.
	\end{itemize}
    Since the dual variables $(P,Q)$ live in a finite-dimensional vector space, it follows that the sublevel set is compact.	
	Since the objective function $J_{\nu,\lambda}$ is lower-semicontinuous, the existence of a minimizer is guaranteed by the extreme value theorem of Weierstrass.

	Next, with the condition $\nu\geq\frac{d}{2}+1$, \cref{lem_boundary} states that the minimizer cannot fall on the boundary of the feasible set. In other words, the optimal $(\hat{P},\hat{Q})$ belong to the interior $\Pfrak_{+,o}\times\Pfrak_{+}$, where the regularized dual function is strictly convex (\cref{lem_convex}). Therefore, such a minimizer is necessarily unique. Moreover, it must satisfy the stationarity conditions (which  are obtained by imposing the first-order differentials in \cref{reg_J_diff} to vanish):
	\alg{\label{part_Q} 
		c_\kb = \int_{\Tbb^d} e^{i\innerprod{\kb}{\thetab}} \left(\frac{P}{Q}\right)^{\nu}\d\m \; \;  \forall\kb\in\Lambda,}
	\begin{align}
	\label{part_P} 
	m_\kb = \int_{\Tbb^d} e^{i\innerprod{\kb}{\thetab}} \left(\frac{\nu}{\nu-1}\left(\frac{P}{Q}\right)^{\nu-1}-\frac{\lambda}{P^\nu} \right) \d\m  \; \;  \forall\kb\in\Lambda_0.
	\end{align}
    We conclude that the rational function $\hat{\Phi}=(\hat{P}/\hat{Q})^\nu$
	matches the covariances $\{c_\kb\}_{\kb \in\Lambda}$ exactly and the $\nu$-cepstral coefficients $\{m_\kb\}_{\kb \in\Lambda}$ approximately with an error term 
	$\varepsilon_\kb := \lambda
	\int_{\Tbb^d} e^{i\innerprod{\kb}{\thetab}} \frac{1}{{\hat P}^{\nu}}\d\m$
	which depends on the regularization parameter $\lambda$.
\end{proof}

To summarize, given $d\geq 3$, we can \emph{always} take the integer parameter $\nu\geq \frac{d}{2}+1$, so that the rational function $\hat \Phi_\nu=(\hat P/\hat Q)^{\nu}$, where $(\hat P,\hat Q)\in \Pfrak_{+,o}\times \Pfrak_+$ is the unique solution to the regularized dual problem \cref{reg_dual_pb}, 
``solves'' the primal problem \cref{primal_pb} in the sense that $\hat \Phi_\nu$ matches $\cb$ and approximately $\mb$, see \cref{part_Q} and \cref{part_P}, respectively. It is worth pointing out that the type of approximate generalized cepstral matching one can do depends on the dimension $d$ since we need to choose the parameter $\nu$ sufficiently large relative to $d/2$.

\begin{remark}

A natural question on the regularization parameter $\lambda$ is: does the solution to the regularized dual problem \cref{reg_dual_pb} converge to that of the unregularized version as $\lambda\to 0$? Such a question seems quite subtle because the dual problem loses strict convexity if we take $\lambda=0$. Indeed one can show along the lines of the proof of \cref{lem_convex} that the Hessian is only \emph{positive semidefinite} in that case so in principle, the dual problem could admit multiple solutions. A careful study (which is left for future work) needs to be carried out for the limiting behavior of the optimal solution as $\lambda$ tends to zero.
\end{remark}

\section{Well-posedness}\label{sec:well-posedness}

On the basis of \cref{thm_main}, we can further show that the regularized dual problem \cref{reg_dual_pb} is well-posed in the sense of Hadamard with respect to the \emph{data} $(\cb,\mb)$. {Moreover, the solution also depends continuously on the regularization parameter $\lambda$}. More precisely, given covariances $\cb$ satisfying \cref{assump_feasible}, $\nu$-cepstral coefficients $\mb$, a regularization parameter $\lambda>0$, and an integer $\nu\ge \frac{d}{2}+1$, \cref{thm_main} already states that the solution $(\hat{P},\hat{Q})$ to \cref{reg_dual_pb}, or equivalently, the coefficient vector of the polynomials $(\hat{\pb},\hat\qb)$, exists and is unique. Hence, to complete the proof of well-posedness, it remains to show the continuous dependence of the solution $(\hat{\pb},\hat\qb)$ on $(\cb,\mb)$.
In fact, we will show in this section that such dependence is even \emph{smooth}, that is, of class $C^\infty$ using the classical inverse and implicit function theorems. To this end, we need first to introduce the set of covariances corresponding to spectral densities of real random fields
\begin{equation}
\begin{aligned}
\Cscr_+ := \left\{ \cb=\Gamma(\Phi_0) : \Phi_0 \text{ satisfies the requirement in \cref{assump_feasible}} \right. \\
\left.\text{and }\Phi_0(e^{-i\thetab}) = \Phi_0(e^{i\thetab}) \right\}
\end{aligned}
\end{equation}
where the linear operator $\Gamma$ has been defined in \cref{Gamma_operator}. Moreover, the set of $\nu$-cepstral coefficients is simply
\begin{equation}
\Mscr := \{\mb\in\Rbb^{|\Lambda_0|} : m_{-\kb} = m_{\kb},\ \kb\in\Lambda_0\}.
\end{equation}
We are then ready to write down the so called moment mapping in a parametric form as suggested by the stationarity conditions of the regularized dual function in \cref{part_Q} and \cref{part_P}:
\begin{equation}\label{f_map}
\begin{aligned}
\fb :\  & \Pfrak_{+,o} \times \Pfrak_{+} \times \Rbb_{+} \to -\Cscr_+ \times \Mscr \\
 & \bmat \pb \\ \qb \\ \lambda \emat \mapsto \bmat -\cb \\ \mb \emat : = \bmat -\Gamma\left(\left(\frac{P}{Q}\right)^\nu\right) \\ \Gamma_0\left(\frac{\nu}{\nu-1}\left(\frac{P}{Q}\right)^{\nu-1} - \frac{\lambda}{P^\nu}\right) \emat ,
\end{aligned}
\end{equation}
where the set $-\Cscr_+ := \{-\cb : \cb\in\Cscr_+\}$. Clearly, the Lagrange multipliers $(\pb,\qb)$ and the data $(\cb,\mb)$ live in the same vector space. Moreover, if both $(\cb,\mb)\in\Cscr_+\times\Mscr$ and $\lambda>0$ are fixed, then we have
\begin{equation}\label{relation_f_J}
D J_{\nu,\lambda}(\pb,\qb) = \fb(\pb,\qb,\lambda) - (-\cb,\mb),
\end{equation}
where the differential operator $D$ applied to $J_{\nu,\lambda}$ is understood as the \emph{gradient} of the regularized dual function, and here with a slight abuse of notation we treat $J_{\nu,\lambda}$ as a function of $(\pb,\qb)$. It then follows that the stationary-point equation of the regularized dual function is equivalent to the system of nonlinear equations
\begin{equation}\label{parametric_moment_eqns}
\fb(\pb,\qb,\lambda) = (-\cb,\mb).
\end{equation}
The result of this section is phrased as follows.

\begin{theorem}\label{thm_wellpos}
	Fix an integer $\nu\ge \frac{d}{2}+1$. Then under \cref{assump_feasible}, we have:
	\begin{enumerate}
		\item For each fixed $\lambda>0$, the map $\omega(\pb,\qb) := \fb(\pb,\qb,\lambda)$ is a diffeomorphism between $\Pfrak_{+,o} \times \Pfrak_{+}$ and $-\Cscr_+ \times \Mscr$ {so the regularized dual optimization problem \cref{reg_dual_pb} is well-posed with respect to the covariance and the $\nu$-cepstral data};
		\item For each fixed vector $(-\cb,\mb)\in -\Cscr_+ \times \Mscr$, the solution $({\hat\pb}_{\lambda},{\hat\qb}_{\lambda})$ to \cref{parametric_moment_eqns} depends smoothly on $\lambda$.
	\end{enumerate}
\end{theorem}

\begin{proof}
To prove the first claim, note first that $\omega$ is a bijection by \cref{thm_main}.
{In particular, surjectivity follows since for each $(-\cb,\mb) \in -\Cscr_+ \times \Mscr$ there exists an optimal solution $(\hat{P}, \hat{Q})$ to the regularized dual problem \cref{reg_dual_pb} such that the equation $\omega(\hat\pb, \hat\qb) = (-\cb, \mb)$ represents the stationarity condition. In addition, injectivity holds since the optimal $(\hat\pb, \hat\qb)$ is unique\footnote{{If two elements in $\Pfrak_{+,o}\times \Pfrak_+$ were mapped to the same $(-\cb, \mb)\in-\Cscr_+ \times \Mscr$, they would both be stationary and hence optimal for \cref{reg_dual_pb} with these covariance and $\nu$-cepstral data, which means that the two elements must coincide.}}. It now remains to show the smoothness of $\omega$ and its inverse.
}

Clearly, $\fb$ is a smooth function on its domain so the smoothness of $\omega$ follows. It now remains to show the smoothness of the inverse $\omega^{-1}$. To this end, we notice that $D\omega(\pb,\qb) = D_{\pb,\qb}\,\fb(\pb,\qb,\lambda) = D^2 J_{\nu,\lambda}(\pb,\qb)$ where the second equality follows from \cref{relation_f_J}. The latter can be seen as the \emph{Hessian} of the regularized dual function $J_{\nu,\lambda}$ which is positive definite (by \cref{lem_convex}) and in particular, nonsingular everywhere. Using the inverse function theorem, we conclude that $\omega$ is a local diffeomorphism which plus the fact that $\omega$ is bijective, leads to the smoothness of $\omega^{-1}$, and hence the first claim.

The second claim is a direct consequence of the implicit function theorem. To see this, consider solving \cref{parametric_moment_eqns}	whose right-hand side is fixed. Clearly, the solution {$(\hat\pb,\hat\qb)$} depends on $\lambda>0$. Since $D_{\pb,\qb}\,\fb(\pb,\qb,\lambda)$
is nonsingular by the previous reasoning, the solution {$(\hat\pb_{\lambda},\hat\qb_{\lambda})$} is locally a smooth function of $\lambda$, and the second claim follows.
\end{proof}

\section{The periodic case}\label{sec:circulant}

In practice, one needs to discretize the multidimensional integrals in the optimization problems  \cref{dual_pb} or \cref{reg_dual_pb} in order to solve them numerically. In particular, the fast Fourier transform (FFT) can be used to compute the Fourier integrals on a discrete grid.
Such a discretization in the frequency domain has an interesting interpretation in the time domain (or more precisely, the ``space'' domain). Indeed, it corresponds to considering a \emph{periodic} stationary random field as explained in {\cite{LPcirculant-13,ZFKZ2019M2}}. 

Let $\Nb$ denote the vector $(N_1,N_2,\dots,N_d)\in\Nbb_+^d$, where the component $N_j$ stands for the number of equal-length partitions of the interval $[0, 2\pi)$ in the $j$-th dimension of $\Tbb^d$. Next define the index set
$\Zbb^d_\Nb:=\left\{ \lb=(\ell_1,\dots,\ell_d) : 0\leq \ell_j \leq N_j-1,\, j=1,\dots,d \right\}$
whose cardinality is $|\Nb|:=\prod_{j=1}^{d}N_j$.
We can now discretize the domain $\Tbb^d$ as
$\Tbb^d_\Nb:=\left\{ \left( \frac{2\pi}{N_1}\ell_1,\dots,\frac{2\pi}{N_d}\ell_d \right) : \lb\in\Zbb_\Nb^d \right\}$,
so the number of grid points is also $|\Nb|$.
Moreover, let us introduce the symbol $\zetab_{\lb}:=\left(\zeta_{\ell_1},\dots,\zeta_{\ell_d}\right)$ for a point in the discrete $d$-torus with $\zeta_{\ell_j}=e^{i2\pi\ell_j/N_j}$ and define $\zetab^{\kb}_{\lb}:=\prod_{j=1}^{d}\zeta_{\ell_j}^{k_j}$ which is just the complex exponential function $e^{i\innerprod{\kb}{\thetab}}$ evaluated at a grid point in $\Tbb_\Nb^d$.
We can now define a discrete measure with equal mass on the grid points in $\Tbb^d_\Nb$ as
\begin{equation}\label{discrete_measure}
\d\eta(\thetab) = \sum_{\lb\in\Zbb^d_\Nb} \delta ( \theta_1-\frac{2\pi}{N_1}\ell_1,\dots,\theta_d-\frac{2\pi}{N_d}\ell_d ) \prod_{j=1}^{d}\frac{\d\theta_j}{N_j}.
\end{equation}
The discrete version of the optimization problem \cref{primal_pb} results when we replace the normalized Lebesgue measure $\d\mu$ with the discrete measure $\d\eta$ and require the discrete spectral density $\Phi$ to be nonnegative on the grid.

Following the analysis in \cite{Zhu-Zorzi-21}, the main result for the discrete problem is summarized as follows.
\begin{theorem}\label{thm_disc_formulation}
	Assume that there exists a function $\Phi_0$ defined on the discrete grid $\Tbb_\Nb^d$ such that $\Phi_0(\zetab_{\lb})>0$ for all $\lb\in\Zbb_\Nb^d$ and the given covariances admit a representation
	$c_\kb = \frac{1}{|\Nb|} \sum_{\lb\in\Zbb^d_\Nb} \zetab_{\lb}^{\kb} \Phi_0(\zetab_{\lb})$ for all $\kb\in\Lambda$.
	Let $J_{\nu}^{\Nb}$ be the discrete version of the function $J_\nu$ in \cref{dual_func} with $\nu\geq 2$.
	Then the optimization problem	
	\begin{equation}\label{dual_pb_discrete}
	\begin{aligned}
	& \underset{P,Q}{\min}
	& & J_{\nu,\lambda}^{\Nb}(P,Q) := J_{\nu}^{\Nb}(P,Q)+\frac{\lambda}{\nu-1} \int_{\Tbb^d}\frac{1}{P^{\nu-1}}\d\eta  \\
	& \text{s.t.}
	& & Q(\zetab_{\lb})>0 \text{ and } P(\zetab_{\lb})>0 \quad \forall \lb\in\Zbb_\Nb^d
	\end{aligned}
	\end{equation}
	admits a unique solution $(\hat{P},\hat{Q})$ in the interior of the feasible set such that the discrete spectral density $\hat{\Phi}_{{\nu}}=(\hat{P}/\hat{Q})^\nu$ achieves covariance matching and approximate $\nu$-cepstral matching:
	\alg{c_\kb&= \frac{1}{|\Nb|}\sum_{\lb\in\Zbb^d_\Nb} \zetab^{\kb}_{\lb} \hat\Phi_{{\nu}}(\zetab_{\lb})  \quad \forall \kb\in\Lambda, \label{cov_match_disc}\\
	m_\kb+\varepsilon_\kb&= \frac{\nu}{\nu-1} \frac{1}{|\Nb|}\sum_{\lb\in\Zbb^d_\Nb}\zetab^{\kb}_{\lb} \hat\Phi_{{\nu}}(\zetab_{\lb})^{\frac{\nu-1}{\nu}}  \quad \forall \kb\in\Lambda_0 \label{cep_match_disc},}
	where the approximation errors in the cepstral matching are 
	$\varepsilon_\kb=  \frac{\lambda}{|\Nb|}\sum_{\lb\in\Zbb^d_\Nb}\zetab^{\kb}_{\lb}\frac{1}{{\hat P}(\zetab_{\lb})^{\nu}}$.
 \end{theorem}

\begin{remark}\label{rem_discrete_formulation}
	Notice that unlike its continuous counterpart, we do not require $\nu\geq \frac{d}{2}+1$. Indeed, in the periodic case if the optimal $\hat P$ is such that $\hat P(\zetab_\ell)>0$ then we have also $\hat Q(\zetab_\ell)>0$. It is worth noting that in the case $\nu = 1$, Problem \cref{dual_pb_discrete} coincides with the one analyzed in \cite{ringh2015multidimensional}.
\end{remark}

Next we state a qualitative result about the discrete solution as an approximation to the continuous counterpart as the grid size goes to infinity. Before that, we need to clarify a number of facts concerning the feasibility of the discrete and continuous optimization problems. These results are somewhat well-known but have been scattered in the literature. We collect them here in the form of {two} lemmas in order to pave the way for our convergence theorem.

In \cite{RKL-16multidimensional}, the feasibility of the continuous optimization problem is stated in terms of the dual cone of the cone of nonnegative polynomials $\overline{\Pfrak}_+$, defined as the closure of set 
\begin{equation}
\Cfrak_+ := \left\{ \cb : \innerprod{\cb}{\qb}>0 \text{ for all } \qb \text{ such that } Q \in \overline{\Pfrak}_+\backslash \{\zerob\} \right\},
\end{equation}	
and written as $\overline{\Cfrak}_+$.
Similarly for the discrete optimization problem, let us define
\begin{equation}
\Cfrak_+(\Nb) := \left\{ \cb : \innerprod{\cb}{\qb}>0 \text{ for all nonzero } \qb \text{ such that } Q(\zetab_{\lb})\geq 0 \  \forall\lb\in\Zbb_\Nb^d \right\}.
\end{equation}

We have the following result.

\begin{lemma}\label{fact_dual_cone}	
	The condition $\cb\in\Cfrak_+$ is equivalent to \cref{assump_feasible}. Similarly, the condition $\cb\in\Cfrak_+(\Nb)$ is equivalent to the assumption at the beginning of \cref{thm_disc_formulation}.
\end{lemma}

\begin{proof}
See \cite[pp.~15, 17]{KreinNudelman} for the unidimensional case and \cite[Proposition 1, p.~1059]{Georgiou-06} for the multidimensional case.
\end{proof}

\begin{remark}
	
In the unidimensional case with the index set $\Lambda=\{-n,\dots,0,\dots,n\}$, the feasibility condition can be reduced to the simple algebraic criterion that the symmetric Toeplitz matrix
\begin{equation}
\bmat c_0&c_1&\cdots& c_n\\
c_1&c_0&\cdots&c_{n-1} \\
\vdots&\ddots&\ddots&\vdots\\
c_n&\cdots&c_1&c_0
\emat
\end{equation}
is positive definite {(see e.g., the classic text \cite{Grenander_Szego})}.
\end{remark}

{The next lemma is taken from \cite{RKL-16multidimensional}. Its interpretation is such that when the grid size $\Nb$ is sufficiently large, the covariance data $\cb$ which are feasible for the continuous problem, are also feasible for the discrete problem.}

\begin{lemma}[Lemma~6.4, \cite{RKL-16multidimensional}]\label{fact_cont_disc_feasiblity}
For any $\cb\in\Cfrak_+$ there exists an integer $N_0$ such that $\cb\in\Cfrak_+(\Nb)$ for all $\min(\Nb)\geq N_0$.
\end{lemma}

Also, we attach a subscript to the discrete measure in \cref{discrete_measure} and write it as $\d\eta_{\Nb}$ to denote its dependence on the grid size $\Nb$.
With the preparations above, we are now ready to state the convergence result.

\begin{theorem}\label{thm_converg_disc_cont}
	Suppose that the covariance and $\nu$-cepstral data $(\cb,\mb)$ are given such that \cref{assump_feasible} holds. Fix an integer $\nu\geq \frac{d}{2}+1$ and the regularization parameter $\lambda>0$. Let $(\hat{P}_{\Nb}, \hat{Q}_{\Nb})$ be the unique solution to the discrete regularized dual problem \cref{dual_pb_discrete}, and let $(\hat{P}, \hat{Q})$ be the solution to the corresponding continuous problem \cref{reg_dual_pb}. Then {the sequence of polynomial coefficients $(\hat{\pb}_{\Nb}, \hat{\qb}_{\Nb})$ is bounded in norm and has a unique limit point $(\hat{\pb}, \hat{\qb})$.}
\end{theorem}

\begin{proof}
	Referring to \cref{fact_dual_cone} and \cref{fact_cont_disc_feasiblity}, we assume that $\min(\Nb)$ is sufficiently large in order to guarantee the feasibility for both discrete and continuous optimization problems.
	Since $(\hat{P}_{\Nb}, \hat{Q}_{\Nb})$ solves \cref{dual_pb_discrete}, we have $J_{\nu,\lambda}^{\Nb}(\hat P_\Nb, \hat Q_\Nb) \leq J_{\nu,\lambda}^{\Nb}(\oneb, \oneb) = \frac{1+\lambda}{\nu-1}+c_0-m_0 =: C_1$ where $\oneb$ denotes the constant function with a value one. Employing an argument similar to that in the proof of \cref{lem_unbounded}, one can show that if $\|(\hat{\pb}_\Nb, \hat{\qb}_\Nb)\|\to \infty$ (which necessarily implies that $\|\hat{\qb}_\Nb\|$ diverges) as $\min(\Nb)$ increases, then $\innerprod{\hat{\qb}_\Nb}{\cb}\to\infty$.
	Based on these two observations, we can use a discrete version of \cref{J_nu_lambda_p1} to conclude that the sequence $(\hat{\pb}_{\Nb}, \hat{\qb}_{\Nb})$ must be bounded. Therefore, it has a convergent subsequence which we still write as $(\hat{\pb}_{\Nb}, \hat{\qb}_{\Nb})\to(\hat{\pb}_\infty, \hat{\qb}_\infty)$. The limit polynomials $\hat{P}_\infty, \hat{Q}_\infty$ are \emph{nonnegative} because: (i) $\hat{P}_\Nb, \hat{Q}_\Nb$ are nonnegative on the grid $\Tbb_\Nb^d$, (ii) the grid becomes dense in $\Tbb^d$ as $\min(\Nb)\to\infty$, and (iii) the subsequence $(\hat{P}_\Nb, \hat{Q}_\Nb)$ converges uniformly to $(\hat{P}_\infty, \hat{Q}_\infty)$.
	
	Now, let us view the nonnegative measures $(\hat P_{\Nb}^{\nu}/\hat Q_{\Nb}^{\nu-1}) \d\eta_{\Nb}$, $(1/\hat P_{\Nb}^{\nu-1}) \d\eta_{\Nb}$ as \emph{linear functionals} over the Banach space of continuous functions on $\Tbb^d$. Due to the fact that $J_{\nu,\lambda}^{\Nb}(\hat P_\Nb, \hat Q_\Nb) \leq C_1$ holds for all $\Nb$, all these functionals are bounded, or equivalently, the corresponding measures have bounded \emph{total variations}. According to the Banach--Alaoglu theorem, there exists a subsubsequence of measures $(\hat P_{\Nb}^{\nu}/\hat Q_{\Nb}^{\nu-1}) \d\eta_{\Nb}$ (the same for $(1/\hat P_{\Nb}^{\nu-1}) \d\eta_{\Nb}$) that converges in the weak* topology to some measure with a bounded total variation. Observe that the rational functions $\hat P_{\Nb}^{\nu}/\hat Q_{\Nb}^{\nu-1}$ converge to $\hat P_{\infty}^{\nu}/\hat Q_{\infty}^{\nu-1}$ in any compact subset of $\Tbb^d\backslash\Zcal(\hat{Q}_\infty)$ where the zero set of $\hat{Q}_\infty$ is understood in the manner of \cref{zero_set_P}. Therefore, the weak* limit of $(\hat P_{\Nb}^{\nu}/\hat Q_{\Nb}^{\nu-1}) \d\eta_{\Nb}$ must have the form $(\hat P_{\infty}^{\nu}/\hat Q_{\infty}^{\nu-1}) \d\mu + \d\hat\gamma_1$ where $\d\hat\gamma_1$ is a nonnegative measure supported in $\Zcal(\hat{Q}_\infty)$, i.e., it is \emph{singular} with respect to the normalized Lebesgue measure $\d\mu$ and satisfies the equality $\int_{\Tbb^d} \hat{Q}_{\infty} \d\hat{\gamma}_1 = 0$. Similarly, we have
	\begin{equation}
	(1/\hat P_{\Nb}^{\nu-1}) \d\eta_{\Nb} \xrightarrow{\text{weak*}} (1/\hat P_{\infty}^{\nu-1}) \d\mu + \d\hat\gamma_2 \,\text{ with }\, \d\hat\gamma_2\geq 0 \,\text{ such that }\, \int_{\Tbb^d} \hat{P}_{\infty} \d\hat{\gamma}_2 = 0.
	\end{equation}
    By the definition of the weak* convergence, we have
    \begin{equation}
    J_{\nu,\lambda}^{\Nb}(\hat P_\Nb, \hat Q_\Nb) \to J_{\nu,\lambda}(\hat P_\infty, \hat Q_\infty) + \frac{1}{\nu-1} \int_{\Tbb^d}\d\hat{\gamma}_1 + \frac{\lambda}{\nu-1}\int_{\Tbb^d}\d\hat{\gamma}_2 \geq J_{\nu,\lambda}(\hat P_\infty, \hat Q_\infty).
    \end{equation}
	Due to fact that the sequence of real numbers $J_{\nu,\lambda}^{\Nb}(\hat P_\Nb, \hat Q_\Nb)$ is bounded by $C_1$, the inequality $J_{\nu,\lambda}(\hat P_\infty, \hat Q_\infty)\leq C_1$ holds. By \cref{lem_boundary}, the argument $(\hat P_\infty, \hat Q_\infty)$ of the regularized dual function (continuous version) cannot be on the boundary of the feasible set since otherwise the function value would be infinity, which would lead to a contradiction. Hence, we conclude that $(\hat P_\infty, \hat Q_\infty)\in \Pfrak_{+,o}\times \Pfrak_{+}$, namely both polynomials are strictly positive over $\Tbb^d$. At the same time, the singular measures $\d\hat{\gamma}_1$ and $\d\hat{\gamma}_2$ do not appear.

    Next, we move to the moment equations which give the stationarity conditions of the (discrete and continuous) regularized dual functions. Consider the discrete version $\fb_\Nb$ of the $\fb$ map in \cref{f_map}, and define similarly the map $\omega_\Nb(\pb,\qb) := \fb_\Nb(\pb,\qb,\lambda)$ for $\lambda$ fixed. It then follows from the optimality condition that $\omega_\Nb({\hat\pb_\Nb,\hat\qb_\Nb}) = (-\cb,\mb)$ and $\omega({\hat\pb,\hat\qb}) = (-\cb,\mb)$.
	Define the vector $(-\cb_\Nb,\mb_\Nb) := \omega({\hat\pb_\Nb,\hat\qb_\Nb})$ which is legitimate because in the latter integrals {at least a subsequence of} the polynomials $\hat{P}_{\Nb}$ and $\hat{Q}_{\Nb}$ will eventually become strictly positive in $\Tbb^d$ when $\min(\Nb)$ is sufficiently large, {as discussed in the previous paragraph}. 
	Next observe the relation
	$\omega_\Nb({\hat\pb_\Nb,\hat\qb_\Nb}) \to \omega({\hat\pb_\Nb,\hat\qb_\Nb})$ as the convergence of the Riemann sum to the integral, so that $(-\cb_\Nb,\mb_\Nb) \to (-\cb,\mb)$ as $\min(\Nb)\to\infty$.
	By \cref{thm_wellpos}, the map $\omega$ is a diffeomorphism so we have
	\begin{equation}\label{converg_poly_coeff}
	{(\hat\pb_\Nb,\hat\qb_\Nb)} = \omega^{-1}(-\cb_\Nb,\mb_\Nb) \to \omega^{-1}(-\cb,\mb) = {(\hat\pb,\hat\qb)}   \ \text{ as } \  \min(\Nb)\to\infty.
	\end{equation}
	Since we already have $(\hat{\pb}_{\Nb}, \hat{\qb}_{\Nb})\to(\hat{\pb}_\infty, \hat{\qb}_\infty)$ in the first part of the proof, it must happen that $(\hat{\pb}_\infty, \hat{\qb}_\infty)=(\hat\pb,\hat\qb)$.	
	
Finally, suppose that $(\hat{P}'_{\infty}, \hat{Q}'_{\infty})$ is another limit point of the sequence $(\hat{P}_{\Nb}, \hat{Q}_{\Nb})$. Then the above reasoning can be repeated to show that such a limit point must coincide with $(\hat{P},\hat{Q})$. This shows the uniqueness of the limit point and completes the proof.
\end{proof}

{Next we state a convergence result for the corresponding spectral density.}

\begin{corollary} 
Under the assumptions of the previous theorem, let {$\hat \Phi_\nu = (\hat{P}/\hat{Q})^\nu$} be the solution of Problem \cref{primal_pb} and $\hat \Phi_{\nu, \Nb} = {(\hat{P}_{\Nb}/\hat{Q}_{\Nb})^\nu}$  be the solution of the corresponding periodic problem with period $\Nb$. Then, {there exists a subsequence of the discrete measures $\hat \Phi_{\nu, \Nb}\d\eta_{\Nb}$ that} tends to $\hat\Phi_\nu \d \mu$ in weak* as $\min(\Nb)\rightarrow\infty$.
\end{corollary}

\begin{proof}
	By \cref{thm_converg_disc_cont}, {there exists a subsequence $(\hat{P}_{\Nb}, \hat{Q}_{\Nb})$ such that the rational function $\hat \Phi_{\nu, \Nb}$ becomes strictly positive for $\min(\Nb)$ sufficiently large. Then it is not difficult to show that $\hat \Phi_{\nu, \Nb}$ converges to $\hat{\Phi}_\nu$ uniformly as $\min(\Nb)\rightarrow\infty$ given the convergence in \cref{converg_poly_coeff}}.
	Moreover, for any continuous function $f(\thetab)$ on $\Tbb^d$, we have the convergence
	$\int_{\Tbb^d} f\, {\hat \Phi_{\nu, \Nb}} \d\eta_{\Nb} \to \int_{\Tbb^d} f\, {\hat\Phi_\nu} \d\m$ as $\min(\Nb)\to\infty$.
	In other words, the measures $\hat \Phi_{\nu, \Nb}\d\eta_{\Nb}$ converge in weak* to $\hat\Phi_\nu \d \mu$. 
\end{proof}

The above corollary states that if $\nu\geq\frac{d}{2}+1$ then we can compute an arbitrarily good approximation of Problem \cref{primal_pb} by solving the corresponding periodic problem on a sufficiently large grid; in the case $\nu<\frac{d}{2}+1$ it is not guaranteed that the discrete solution represents an approximate solution to Problem \cref{primal_pb}, e.g., the case with $d\geq 3$ and $\nu=1$ {or $2$}.

\section{Numerical example: spectral estimation of a $3$-d random field}\label{sec:sysid_app}

Following the discrete formulation described in \cref{sec:circulant}, here we present numerical experiments on the {spectral estimation of a three-dimensional ($d=3$) stationary random field generated by a linear shaping filter with a white noise input}. 
Throughout this section, we set $\nu=3$ which is taken as an oracle parameter assumed to be known in the {estimation} procedure. Notice that we have chosen the parameters to respect the condition $\nu\geq \frac{d}{2}+1$.

Let $W(z_1,z_2,z_3)$ be the transfer function of a $3$-d linear time-invariant (LTI) system, which is excited by a white noise process/field $e(t_1,t_2,t_3)$ and produces an output process $y(t_1,t_2,t_3)$.
Furthermore, let the true system $W$ have the structure that corresponds to our optimal spectrum \cref{Phi_nu}, i.e., our estimation paradigm considers the most parsimonious model structure containing the actual model:
\begin{equation}\label{transfer_func_W}
W(\zb) = \left[ \frac{b(\zb)}{a(\zb)} \right]^\nu = \left[ \frac{b_0 - b_1z_1^{-1} - b_2z_2^{-1} - b_3z_3^{-1}}{a_0 - a_1z_1^{-1} - a_2z_2^{-1} - a_3z_3^{-1}} \right]^\nu
\end{equation}
where $(z_1,z_2,z_3)$ is abbreviated as $\zb$, so that the \emph{degree-one} polynomials $a(\zb)$ and $b(\zb)$ are characterized by their coefficient vectors $\ab=[a_0,\dots,a_3]$ and $\bb=[b_0,\dots,b_3]$, respectively\footnote{Such a model can equivalently be described in the time domain by a $3$-d autoregressive moving-average (ARMA) recursion.}. 
If the white noise input $w$ has unit variance, then the spectral density of the output process $y$ is $(P/Q)^\nu$ where
$P(e^{i\thetab}) = |b(e^{i\thetab})|^2$ and $Q(e^{i\thetab}) = |a(e^{i\thetab})|^2$ are two symmetric polynomials in $\overline{\Pfrak}_+$.
Next we identify the index set $\Lambda$ that has gone through the analysis of previous sections.
Let
$\Lambda_+ := \{\, (0,0,0),\, (1,0,0),\, (0,1,0),\, (0,0,1) \,\}$
be the index set corresponding to $a(\zb)$ and $b(\zb)$. Clearly, we have the relation\footnote{The set difference in \cref{ex_Lambda_set} is understood as $A-B:=\{x-y : x\in A,\ y\in B\}$.} 
\begin{equation}\label{ex_Lambda_set}
\Lambda = \Lambda_+ - \Lambda_+ = \Lambda_{\half} \cup (-\Lambda_{\half})
\end{equation}
where we can choose e.g.,
$\Lambda_{\half} = \Lambda_+ \cup \{\, (1,-1,0),\, (1,0,-1),\, (0,1,-1) \,\}$.
Due to the symmetry $q_{-\kb}=q_{\kb}$ and $p_{-\kb}=p_{\kb}$, the optimization variables of the dual problem \cref{dual_pb} or \cref{dual_pb_discrete} in this particular example can be identified with $(\pb,\qb)$ where
\begin{equation}\label{dual_variables_ex}
\pb=\{p_\kb : \kb\in\Lambda_{\half}\backslash\{\zerob\}\} \text{ and }  \qb = \{q_{\kb} : \kb\in\Lambda_{\half}\}, 
\end{equation}
so that the total number of variables is $13$.
Moreover, we see that the normalization condition $p_\zerob=1$ for the numerator polynomial $P$ translates into $\|\bb\|^2=1$. {For simplicity}, we specify $a_0=1$.

The model \cref{transfer_func_W}, albeit simple, allows convenient pole/zero specification. For example, take $a(\zb) = 1 - a_1z_1^{-1} - a_2z_2^{-1} - a_3z_3^{-1}$ with $a_j = \rho_j e^{i\varphi_j}$, $j=1,2,3$ if we admit complex coefficients. Then we can choose the moduli such that $\rho_1+\rho_2+\rho_3=1$, so that the point $(e^{i\varphi_1},e^{i\varphi_2},e^{i\varphi_3})$ on the $3$-torus, or roughly, the frequency vector $(\varphi_1,\varphi_2,\varphi_3)\in\Tbb^3$ will be a pole of $W$. In what follows, we shall perform simulations on two instances of the true model \cref{transfer_func_W} having the same \emph{real} denominator $a(\zb)$ such that
$\mathtt{abs}(\ab) = [1, 0.3, 0.3, 0.3],\ \mathtt{angle}(\ab) = \zerob$, 
where $\mathtt{abs}$ and $\mathtt{angle}$ are Matlab functions for retrieving the element-wise modulus and phase angle of a complex array. Clearly, the resulting $Q=|a|^2$ is strictly positive on $\Tbb^3$.
The difference between the two true models lies in the numerator polynomial $b(\zb)$, and they are specified as
\begin{subequations}
\begin{align}
\mathtt{abs}(\tilde{\bb}_1) = [1, 0.2, 0.3, 0.4],\ \mathtt{angle}(\tilde{\bb}_1) = [0, \pi, \pi, \pi], \ \bb_1=\tilde{\bb}_1/\|\tilde{\bb}_1\|; \label{model_zeroless} \\ 
\mathtt{abs}(\tilde{\bb}_2) = [1, 0.2, 0.3, 0.5],\ \mathtt{angle}(\tilde{\bb}_2) = [0, \pi, \pi, \pi], \ \bb_2=\tilde{\bb}_2/\|\tilde{\bb}_2\|. \label{model_with_zero} 
\end{align}
\end{subequations}
As a consequence, the symmetric polynomial $P_1=|b_1|^2$ is strictly positive, while $P_2=|b_2|^2$ (hence also the spectrum) has a zero at the frequency vector $[\pi,\pi,\pi]$.

Given a true model \cref{transfer_func_W}, the procedure for {spectral estimation} is summarized below:
\begin{enumerate}
	\item Fix $N=20$ and the grid size $\Nb=[N,N,N]$ for $\Tbb_\Nb^3$. Choose a regularization parameter $\lambda>0$.
	\item Compute the covariances $\{c_\kb : \kb\in\Lambda\}$ and the $\nu$-cepstral coefficients $\{m_\kb : \kb\in\Lambda_0\}$ via \cref{constraint_cov} and \cref{constraint_cep} where $\d\m$ is replaced by $\d\eta$ in \cref{discrete_measure}.
	\item Solve the discrete regularized dual problem \cref{dual_pb_discrete} given $\cb$ and $\mb$. 
	\item Compare the optimal solution $(\hat{\pb},\hat{\qb})$ with the polynomial coefficients $(\pb,\qb)$ computed from the true model parameters $(\ab,\bb)$ and evaluate the error.
\end{enumerate}

In Step 3, the optimization problem is solved using Newton's method. In doing so, we need to compute the gradient and the Hessian of the regularized dual function after a suitable \emph{choice of basis}. Referring to \cref{dual_variables_ex}, we see that the dual variable $(\pb,\qb)$ resides in the Euclidean space $\Rbb^{13}$ after vectorization, and we can conveniently choose the canonical basis $\{e_j\}_{j=1}^{13}$ where the vector $e_j$ has one at the $j$-th component and zeros elsewhere. Each $e_j$ corresponds to a direction (polynomial) $\delta P_j$ and $\delta Q_j$. More precisely,
\begin{equation}
\begin{aligned}
q_{\kb} = 1 & \text{ gives } \delta Q = e^{-i\innerprod{\kb}{\thetab}} + e^{i\innerprod{\kb}{\thetab}},\ \delta P = 0 \text{ for } \kb\in\Lambda_{\half}; \\
p_{\kb} = 1 & \text{ gives } \delta Q = 0,\ \delta P = e^{-i\innerprod{\kb}{\thetab}} + e^{i\innerprod{\kb}{\thetab}} \text{ for } \kb\in\Lambda_{\half}\backslash \{\zerob\}.
\end{aligned}
\end{equation}
With some abuses of notation, we identify the regularized dual function $J_{\nu,\lambda}^{\Nb}$ as function of $(\pb,\qb)$ in \cref{dual_variables_ex}. Then the gradient is just a vector in $\Rbb^{13}$ whose $j$-th entry is
\begin{equation}
\delta J_{\nu,\lambda}^{\Nb}(\pb,\qb;e_j) = \delta J_{\nu,\lambda}^{\Nb}(P,Q;\delta Q_j) + \delta J_{\nu,\lambda}^{\Nb}(P,Q;\delta P_j),
\end{equation}
where the formula for the first-order differential can be adapted from \cref{reg_J_diff} to the discrete case.
Similarly, the Hessian is a $13\times 13$ matrix whose $(j,k)$ element is
\begin{equation}
\delta^2 J_{\nu,\lambda}^{\Nb}(\pb,\qb;e_j,e_k) = \delta^2 J_{\nu,\lambda}^{\Nb} (P,Q; \delta P_j, \delta Q_j, \delta P_k, \delta Q_k),
\end{equation}
where the formula for the second-order differential is given in \cref{reg_J_sec_diff}.
After the Newton direction is computed, the stepsize is determined by the standard backtracking line search, see e.g., \cite{bv_cvxbook}.
It is worth mentioning that evaluation of the complex exponential $e^{-i\innerprod{\kb}{\thetab}}$ {and} the polynomials $a(\zb)$ and $b(\zb)$ in \cref{transfer_func_W} on the grid $\Tbb_\Nb^d$ can be done efficiently via the FFT. This is one of the reasons that the discrete formulation given in \cref{sec:circulant} is appealing.

Next we report the simulation results. For the zeroless model \cref{model_zeroless}, we repeat the {estimation} procedure with a wide range of values of the regularization parameter, namely $\lambda= 10^{-10}, 10^{-8}, 10^{-6}, 10^{-4}, 10^{-2}, 1$, and the errors $\|(\hat{\pb},\hat{\qb}) - (\pb,\qb)\|$ are shown in the blue line of the left panel of \cref{fig:secx_vs_lambda}. It is clear that the error decreases monotonically as $\lambda\to 0$. Another simulation is performed on the model \cref{model_with_zero} with a spectral zero, and the same observation can be made, although the orange error curve goes down more slowly and rests on a higher level than the blue one as $\lambda$ decreases.
Such a result is not so surprising and can be explained by the presence of a spectral zero in the model \cref{model_with_zero}.
Indeed, the latter model does not fall into the model class prescribed by \cref{Phi_nu} where both $P$ and $Q$ are strictly positive and for this reason, regularization is needed in order to construct an approximate model within that model class.

In the right panel of \cref{fig:secx_vs_lambda}, we also compare the values of the reconstructed spectral densities with the true one corresponding to the model \cref{model_with_zero} (with a spectral zero, the orange line in the left panel) along a cross section of the true spectral zero located at the grid index\footnote{Under the Matlab convention, the first array index in each dimension is $1$.} $[11, 11, 11]$ where the true spectrum has a numerical value on the scale of {$10^{-98}$}.
Notice that only one cross section $[\,\cdot \;11\;11\,]$ (along the first dimension) of the spectral densities is shown there, because the other two cross sections $[\,11\;\cdot \;11\,]$ and $[\,11\;11\;\cdot \,]$ are graphically indistinguishable from the presented figure in the logarithmic scale. One can tell that as $\lambda\to 0$, the approximation of the true nonnegative spectrum using strictly positive ones becomes increasingly better, which certainly agrees with the finding of the left panel.

\begin{figure}[!t]
	\centering
	\includegraphics[width=0.47\textwidth]{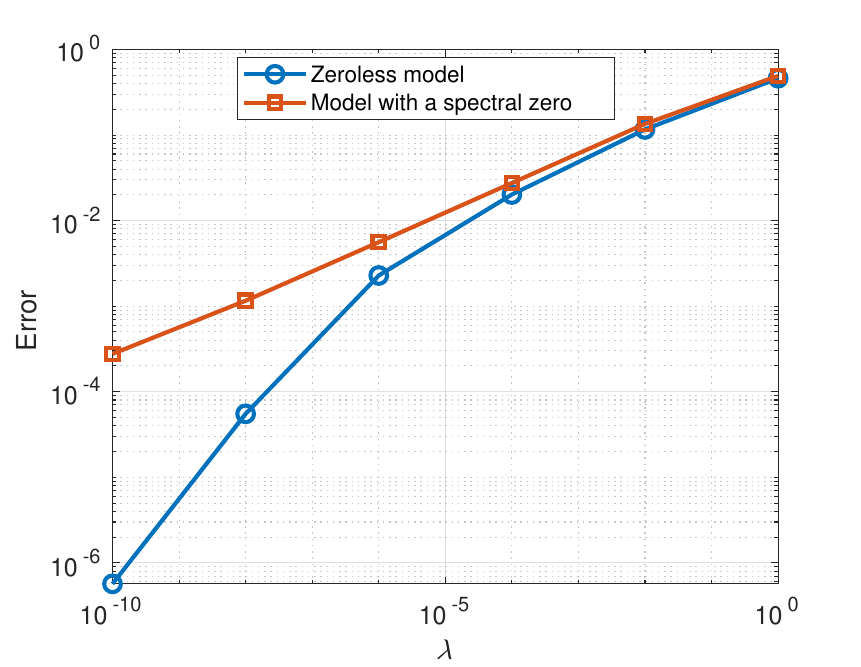}
	\includegraphics[width=0.5\textwidth]{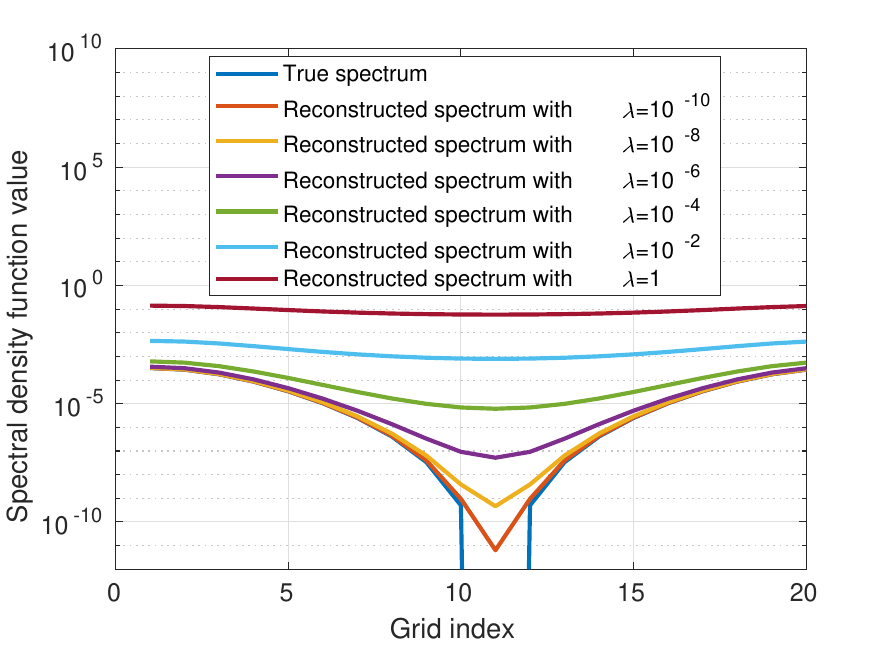}
	\caption{\emph{Left:} Error of spectrum reconstruction versus the regularization parameter $\lambda = 10^{-10}, 10^{-8}, 10^{-6}, 10^{-4}, 10^{-2}, 1$, where both axes are in a logarithmic scale. \emph{Right:} The true spectrum with a zero and the reconstructed spectra with {the previous} choices of $\lambda$ at the cross section $[\,\cdot \;11\;11\,]$, i.e., $\hat{\Phi}_{\nu,\lambda}(e^{i\thetab})$ with $\thetab=2\pi\times[(k-1)/20, 11/20, 11/20]$ for grid indices $k=1,\ldots, 20$, where the vertical axis is in a logarithmic scale. Note that the display of the true spectrum (the blue line) is incomplete because its minimum (at index $11$) is on the scale of {$10^{-98}$}.}
	\label{fig:secx_vs_lambda}
\end{figure}

\section{Conclusions}\label{sec:conclusion}

In this paper, we generalize the classical optimization-based covariance and cepstral extension problem to a sequence of parametrized formulations that is flexible enough to guarantee a rational spectral density as the solution. The notions of the $\nu$-cepstral coefficient and the $\nu$-entropy are utilized in order to bring about the general optimization formulation. Following the lines in our dual analysis, it turns out that regularization is needed in order to enforce a strictly positive spectral density as a solution to the primal problem. In this direction, we propose a new regularization term which is compatible with the $\nu$-entropy, so that the regularized dual problem admits a unique interior-point solution. Furthermore, we show that such a solution depends continuously on the covariances and the $\nu$-cepstral coefficients so that the regularized dual problem is well-posed in the sense of Hadamard.
In order to simplify the numerical computation, we provide a discrete formulation of the optimization problem which can be associated to a periodic random field. It is shown that the discrete solution is a reasonable approximation of the continuous counterpart via a convergence argument as the grid size goes to infinity. 
Numerical simulations in {spectral estimation} reveal that the regularized strictly positive solution of our general optimization problem can well approximate a true model with a spectral zero (hence only nonnegative) when the model class is correctly chosen and the regularization parameter $\lambda$ is small.

\section*{Acknowledgement}

The authors would like to thank the anonymous reviewers for their constructive comments that helped to significantly improve the quality of the paper.

\appendix

\section{Convexity of the function $g(x,y)$ in \cref{func_g} with $\nu\geq 2$ odd}\label{sec:appendix_A}

Consider the bivariate function $g_1(x,y) := x^\nu/y^{\nu-1}$ where $x\in\Rbb$, $y>0$.
It is then straightforward to compute
\begin{equation}\label{derivatives_g_1}
\nabla g_1(x,y) = x^{\nu-1} y^{-\nu} \bmat \nu y \\ (1-\nu) x\emat, \quad 
\nabla^2 g_1(x,y) = \underbrace{\nu(\nu-1)}_{>0} x^{\nu-2} y^{-\nu-1} \bmat y^2 & -xy \\ -xy & x^2 \emat.
\end{equation}
By assumption the odd number $\nu$ is in fact at least $3$, so the factor $x^{\nu-2}$ in the Hessian has the same sign as $x$ (for $x\neq 0$). Therefore, $g_1$ is convex in the region $U_1 := \{x> 0,\ y>0\}$ and \emph{concave} in $U_2 := \{x\leq 0, \ y>0\}$, and the convexity of the composite function $g(x,y)=\max\{0, g_1(x,y)\}$ does not follow in a trivial manner.

We notice that $g$ is a $C^1$ function in the upper half plane $U_1\cup U_2$. Indeed, we only need to establish the continuity of the partial derivatives on the vertical half axis $\{x=0,\ y>0\}$. The latter point is obvious since $g_1$ is smooth on the upper half plane and $\nabla g_1(0,y)=\zerob$. Hence, we can use the first-order derivative test for convexity which is $g(x_1,y_1) \geq g(x_2,y_2) + \nabla g(x_2,y_2)^\top \bmat x_1-x_2 \\ y_1-y_2\emat$ for two arbitrary points $(x_1, y_1)$ and $(x_2,y_2)$ in the upper half plane.
Given a partition $U_1, U_2$ of the domain, taking two points leads to four cases, and the only nontrivial case is $(x_1,y_1)\in U_2$ and $(x_2,y_2)\in U_1$:
\begin{equation}
0 \geq g_1(x_2,y_2) + \nabla g_1(x_2,y_2)^\top \bmat x_1-x_2 \\ y_1-y_2\emat
\iff 0 \geq \nu y_2 x_1 + (1-\nu) x_2y_1.
\end{equation}
The latter inequality holds because $x_1\leq 0$ and $1-\nu<0$. Therefore, $g$ is indeed convex when $\nu$ is a fixed odd number greater than $2$.

\bibliographystyle{siamplain}
\bibliography{references}

\end{document}